\documentclass[pdflatex,sn-mathphys-num,draft]{sn-jnl}% Math and Physical Sciences Numbered Reference Style
\usepackage{graphicx}%
\usepackage{multirow}%
\usepackage{amsmath,amssymb,amsfonts}%
\usepackage{amsthm}%
\usepackage{mathrsfs}%
\usepackage[title]{appendix}%
\usepackage{xcolor}%
\usepackage{textcomp}%
\usepackage{manyfoot}%
\usepackage{booktabs}%
\usepackage{algorithmicx}%
\usepackage{algpseudocode}%
\usepackage{listings}%
\theoremstyle{thmstyleone}%
\newtheorem{theorem}{Theorem}%  meant for continuous numbers
\theoremstyle{thmstyletwo}%
\newtheorem{example}{Example}%
\newtheorem{remark}{Remark}%

\theoremstyle{thmstylethree}%
\newtheorem{definition}{Definition}%

\usepackage{lscape}
\usepackage{multirow}
\usepackage{arydshln}
\usepackage{url,floatflt}
\usepackage{subfig}
\usepackage{rotating}
\usepackage{pifont}
 \usepackage{mathdots}
\usepackage{todonotes}

\usepackage[shortlabels]{enumitem}
\usepackage[table]{xcolor}
\usepackage{tikz}
\usetikzlibrary{shapes,snakes}
\usetikzlibrary{patterns}
\usepackage[ruled,linesnumbered]{algorithm2e}
\usepackage{pdflscape}
\usepackage{mathtools}
 \usepackage{amsmath} 
 
\newcommand{\antidiagonaltranspose}{{\scriptscriptstyle\mathrlap{\rotatebox[origin=c]{-45}{$\hspace{-0.6pt} \scriptscriptstyle\leftrightarrow$}}\diagup}}

\newcommand{\binomseq}[2]{\left\{\!\binom{#1}{#2}\!\right\}}
\newcommand{\textbinomseq}[2]{{\textstyle\big\{\!\binom{#1}{#2}\!\big\}}}

\newcommand{\Fset}{\mathbb{F}}

\DeclareMathOperator{\id}{id}
\DeclareMathOperator{\lcm}{lcm}

\usepackage[normalem]{ulem}

\newtheorem{lemma}[theorem]{Lemma}
\newtheorem{corollary}[theorem]{Corollary}

\begin{document}

\title[Binomial sequences over prime fields]{Binomial sequences over prime fields}

%%=============================================================%%
%% GivenName	-> \fnm{Joergen W.}
%% Particle	-> \spfx{van der} -> surname prefix
%% FamilyName	-> \sur{Ploeg}
%% Suffix	-> \sfx{IV}
%% \author*[1,2]{\fnm{Joergen W.} \spfx{van der} \sur{Ploeg} 
%%  \sfx{IV}}\email{iauthor@gmail.com}
%%=============================================================%%

\author[1]{\fnm{Miguel} \sur{Beltrá}}\email{miguel.beltra@ua.es}
\equalcont{All authors contributed equally to this work.}

\author*[2]{\fnm{Sara} \sur{D. Cardell}}\email{sd.cardell@unesp.br}
\equalcont{All  authors contributed equally to this work.}

\author[1]{\fnm{Verónica} \sur{ Requena}}\email{vrequena@ua.es}
\equalcont{All authors contributed equally to this work.}

\affil*[1]{\orgdiv{Department de Matemátiques}, \orgname{Universitat d'Alacant}, \orgaddress{\street{Carretera San Vicente del Raspeig s/n}, \city{Alicante}, \postcode{03690}, \country{Spain}}}

\affil[2]{\orgdiv{Departmento de Matemática, Instituto de Geociências e Ciências Exatas - IGCE}, \orgname{Universidade Estadual Paulista}, \orgaddress{\street{Av. 24A}, \city{Rio Claro-SP}, \postcode{13506-900}, \country{Brazil}}}

%%==================================%%
%% Sample for unstructured abstract %%
%%==================================%%

\abstract{The binary binomial sequences  correspond to the diagonals of  the Pascal's triangle modulo 2.
  They have interesting properties such as they form a basis of the linear  space of  all binary sequences with period a power of $2$.
  Other properties of these sequences (period, linear complexity, construction rules or relations among different binomial sequences)  have been deeply  analysed in detail previously. 
  In this work, we study the binomial $p$-ary sequences for a prime $p$, its intrinsic characteristic and formation rules. 
  We also prove that the family of $p$-ary sequences with period a power of $p$  form a vector  space over $\mathbb{F}_p$ and that the  family of binomial $p$-ary sequences is a basis of this space.
  }


%%\pacs[JEL Classification]{D8, H51}

%%\pacs[MSC Classification]{35A01, 65L10, 65L12, 65L20, 65L70}

\maketitle

 \section{Introduction}\label{sec:intro}

Sequences generated by linear recursions are widely used in practice due to their periodic properties. The characteristic polynomials that produce sequences with maximal periods \cite{Golomb1982bk,Lidl1986bk} over finite fields have been extensively  analysed in various contexts (see, for instance, \cite{Goltvanitsa2012,Tsaban2002}). More recently, recurrent sequences over rings and finite fields have been explored \cite{Bush2021}, with a focus on decomposing the set of sequences into periodic and null components. Additionally, this work   identified the complete set of possible least periods associated with polynomials of a given degree over a finite field.

Expanding on these investigations, Ganesan \cite{GANESAN2021}  analyses the periodicity structure of linear recurring sequences over finite fields in cases where the period is not necessarily maximal. Specifically, he establishes necessary and sufficient conditions for the characteristic polynomial to allow exactly two possible periods, one being 1 and the other a unique integer greater than 1, across all sequences it generates.

There are several studies on the use of non-binary fields in technology applications. For example, \cite{Bos2024PhDT} explores the ternary case as a more efficient alternative to binary, focusing on circuit design, memory implementation, and logic synthesis tools such as MRCS and the REBEL-2 processor. Historically, practical implementations of ternary computing, such as the Setun and Setun 70 computers developed at Moscow State University \cite{Brousentsov}, demonstrated advantages in efficiency, simplicity, and arithmetic operations. Despite these benefits, binary computing remained dominant due to industrial and economic factors.

In this work, we focus on the study of binomial sequences.
These sequences are characterized as the diagonals of Pascal's triangle. When considered over the binary field, these sequences have been extensively studied, both in terms of their intrinsic properties, such as period and complexity, and their recursive formation rules \cite{Cardell2019a}. Moreover, it has been shown that certain cryptographic (binary) sequences can be expressed as a linear combination of these sequences \cite{Cardell2020c}.
In this work we  analyse the binomial sequences over prime fields. In particular, we study their linear complexity, period, characteristic polynomial and construction rules.
We also prove that the $p$-ary binomial sequences form a basis of the space of $p$-ary sequences with period a power of $p$ over $\mathbb{F}_p$. 
It is worth noting that this type of study is novel, and no related literature has been found regarding the study of these sequences over finite fields
 
The paper is organized as follows: In Section 2,
we introduce some basic concepts about $p$-ary sequences. 
Section 3 recalls the basic properties of the binary binomial sequences. 
In Section 4,  we study the properties of the binominal sequences over prime fields. 
Finally, the paper concludes with some  final remarks and future work.

%%%%%%%%%%%%%%%%%%%%%%%%%%%%%%%%%%%%%%%%%%%%%%%
\section{Preliminaries}\label{sec:prel}
In this section, we present some basic concepts needed to understand the rest of the paper.

%%%%%%%%%%%%%%%%%%%%%%%%%%%%%%%%%%%%%%%%%%
\subsection{Sequences}
%\todosara{Si no vamos a colocar más subsecciones podríamos quitar este título}
Let $\mathbb{F}_p$ be the Galois field of $p$ elements.
We say $\{s_n\}_{n\geq 0}  = \{s_0, s_1, s_2, \ldots\}$ is a \textbf{$p$-ary sequence} if its terms $s_n\in \mathbb{F}_p$, for $n=0, 1, 2,\ldots$. 
In the sequel, all the sequences considered will be $p$-ary sequences and will be denoted by $\{s_n\}$ instead of $\{s_n\}_{n\geq 0} $. 
%We write $\{ s_n \}_{n \geq 0}^N$ to denote the first $N$ terms $\{ s_0,s_1,\ldots,s_N \}$.

We denote by $V(\mathbb{F}_q)$ the set of all $p$-ary sequences. Together with the operations given by $\{s_n\}  + \{ \sigma_n \}  = \{ s_n+\sigma_n \} $ and $c \{ s_n \}  = \{ c s_n \} $, with $c \in \mathbb{F}_p$, the set $V(\mathbb{F}_p)$ forms a linear space over $\mathbb{F}_p$, with the zero sequence $\{ 000 \ldots \}$ as the neutral element for the sum \cite{GolombGong2005}. 

A sequence $\{s_n\} $ is said to be \textbf{periodic} if and only if there exists an integer $N$ such that $s_{n+N}=s_n$, for all $n\geq 0$. The minimum positive integer $T$ such that $s_{n+T} = s_n$, for all $n \geq 0$, is called the \textbf{period} of the sequence. 
From now on, we will work with sequences $\{s_n\}$ as the first $T$ terms $\{ s_0,s_1,\ldots,s_T \}$.
%Notice that if $\{s_n\}$ is of period $T$ and $s_{n+N} = s_{n}$, for all $n \geq 0$, then $T \mid N$, otherwise there exist integers $Q$ and $R$ with $R < T$ such that $N = QT+R$ and $s_n = s_{n+N} = s_{n + (QT+R)} = s_{n+R}$ for all $n \geq 0$, contradicting the fact that $T$ is the period of $\{ s_n \}$}. 
%and the sum modulo $p$ operation among sequences will be denoted by $+$.

Let $r$ be a positive integer and let $d_1, d_2, d_3, \ldots, d_{r}$ be constant coefficients with $d_i\in\mathbb{F}_p$.
A $p$-ary sequence $\{s_n\} $ satisfying the relation
\begin{equation}\label{eq:1}	
s_{n+r}=d_rs_n + d_{r-1}s_{n+1}+ \cdots +d_{3}s_{n+r-3}+d_{2}s_{n+r-2}+d_{1}s_{n+r-1}, \quad n\geq 0,
\end{equation}
is called a ($r$-\textbf{th order}) \textbf{linear recurring sequence} in $\mathbb{F}_p$.
The terms $\{s_0, s_1, \ldots,s_{r-1}\}$ are referred to as the \textbf{initial values} (or \textbf{initial state}) and determine the rest of the sequence uniquely.
A relation of the form given by the equation~(\ref{eq:1}) is called a ($r$-\textbf{th order}) \textbf{linear recurrence relationship}.

We define the \textbf{shifting operator} $\textbf{E}: V(\mathbb{F}_p) \longrightarrow V(\mathbb{F}_p)$ as the map
\[ \textbf{E}\{ s_0, s_1, s_2, \ldots \} = \{ s_1, s_2, s_3, \ldots \}. \]
This is a linear operator over $V(\mathbb{F}_p)$. Let us denote by $E$ the action of $\textbf{E}$ on each term of the sequence, that is, $\textbf{E}\{ s_0, s_1, s_2, \ldots \} = \{ Es_0, Es_1, Es_2, \ldots \}$. Hence, $E s_n = s_{n+1}$ for all $n \geq 0$. We can extend this notation: for all integer $t \geq 0$, we denote $E^t s_n = s_{n+t}$.
Thus, we can rewrite \eqref{eq:1} as
\[
E^r s_n = d_r s_n + d_{r-1} E s_n + \cdots + d_{3} E^{r-3} s_n + d_2 E^{r-2} s_n + d_1 E^{r-1} s_n, \quad n\geq 0.
\]
Hence, if we consider the monic polynomial 
$$ p(x) = x^{r}- \sum_{i=1}^{r}d_ix^{r-i} \in \mathbb{F}_p[x], $$
% \todosara{no seria $d_rx_r$?}
% \todovero{No, pues hace referencia a $E^r s_n$ que sería mónico.}
it happens that
\begin{equation}\label{eq:polynomial_shifting_operator}
    p(E)s_n = 0, \quad n \geq 0.
\end{equation}
A polynomial $p(x) \in \mathbb{F}_p[x]$, not necessarily monic, that satisfies \eqref{eq:polynomial_shifting_operator} is called a \textbf{generator polynomial} of a linear recurring sequence $\{ s_n \} $ and $\{s_n\} $ is said to be \textbf{generated} by $p(x)$.
The constant polynomial $1$ is the characteristic polynomial of the zero sequence $\{0, 0, 0, \ldots\}$, the polynomial $x-1$ is the characteristic polynomial of any constant sequence $\{c, c, c, \ldots \}$ with $c \in \mathbb{F}_{p}\setminus \{0\}$
and the $0$ polynomial generates all the sequences. Given a polynomial, the set of all sequences generated by this  polynomial is a linear subspace  of the linear space of $p$-ary sequences. 
Given a sequence $\{ s_n \} $ there are several polynomials generating it. In fact, the set $I$ of all generator polynomials of a sequence is an ideal of $\mathbb{F}_p[x]$  \cite[Theorem 4.4]{GolombGong2005}. 
Recall that $\mathbb{F}_p[x]$ is a principal ideal domain and hence either $I = \langle 0 \rangle$ or there is a unique monic polynomial $p(x)$ such that $I = \langle p(x) \rangle$ with $\deg(p(x)) < \deg(q(x))$ for any other monic polynomial $q(x) \in I$ \cite[Ch. 1, \S 5, Corollary 4]{CoxLittleOShea2015}. 
The unique nonzero monic generator polynomial of lowest degree of a sequence $\{ s_n \} $ is called its \textbf{characteristic polynomial}.

% The generation of linear recurring sequences can be implemented on LFSRs \cite{Golomb1982bk}.
% These structures handle information in the form of binary elements and they are based on shifts and linear feedback.
% In fact, an LFSR is an electronic device with $r$ memory cells (stages) with binary contents. At each time instant, each element is shifted to the adjacent stage and a new element is computed via a linear feedback to fill the empty stage (see Figure~\ref{fig:LFSR}).
% % If the characteristic polynomial of the linear recurring sequence is primitive \cite{Golomb1982bk}, then the LFSR is a maximal-length LFSR and its output sequence, the so-called PN-sequence, has period $T = 2^{r}-1$.
% % %Such an output sequence is called PN-sequence (pseudonoise sequence).

 The \textbf{linear complexity}, $LC$, of a sequence $\{s_n\} $ is defined as the order of the lowest order linear recurrence relationship that generates such a sequence (the degree of the characteristic polynomial).

\subsection{Matrices}
In this section we are going to recall some special matrices.
\begin{definition}
Given a square matrix $A \in \mathrm{Mat}_{n}(\mathbb{F}_p)$, we define its \textbf{transpose with respect to its antidiagonal}, denoted by $A^\antidiagonaltranspose$, 
as the matrix:
\[
A^\antidiagonaltranspose=
\left[ 
    \begin{array}{ccccc}
        a_{11} & a_{12} & \cdots & a_{1\, n-1} & a_{1n} \\
        a_{21} & a_{22} & \cdots & a_{2,n-1} & a_{2n} \\
        \vdots & \vdots &  & \vdots &\vdots \\
        a_{n-1\, 1} & a_{n-1\, 2} & \cdots & a_{n-1\, n-1} & a_{n-1\, n} \\
        a_{n1} & a_{n2} & \cdots & a_{n\, n-1} & a_{nn}
    \end{array}
\right]^\antidiagonaltranspose
=
\left[ 
    \begin{array}{ccccc}
        a_{nn} & a_{n-1\, n} & \cdots & a_{2n} & a_{1n} \\
        a_{n\, n-1} & a_{n-1\, n-1} & \cdots & a_{2\, n-1} & a_{1\, n-1} \\
        \vdots & \vdots &  & \vdots&\vdots \\
        a_{n2} & a_{n-1\, 2} & \cdots & a_{22} & a_{12} \\
        a_{n1} & a_{n\, n-1} & \cdots & a_{21} & a_{11}
    \end{array}
\right].
\]    
\end{definition}

If we consider $J$ to be the square matrix 
with ones in the antidiagonal and zeros in the rest of the entries, i.e.,
\[
J = 
\left[
    \begin{array}{cccc}
        0 & \cdots & 0 & 1 \\
        0 & \cdots & 1 & 0 \\
        \vdots & & \vdots & \vdots \\
        1 & \cdots & 0 & 0
    \end{array}
\right],
\]
then it is easy to show that $A^\antidiagonaltranspose = J A^t J$. The matrix $J$ is usually known as \textbf{exchange matrix} or \textbf{reversal matrix} \cite{Horn2012bk}. Notice that $J^{-1} = J$.

\begin{definition}
Given two matrices $A\in   \mathrm{Mat}_{n\times m}(\mathbb{F}_p)$ and $B \in \mathrm{Mat}_{r\times s}(\mathbb{F}_p)$, its \textbf{Kronecker product} is the $nr \times ms$ matrix defined as
\[
A \otimes B =
\left[ 
    \begin{array}{cccc}
        a_{11} B & a_{12} B & \cdots & a_{1m} B \\
        a_{21} B & a_{22} B & \cdots & a_{2m} B \\
        \vdots & \vdots &   & \vdots \\
        a_{n1} B & a_{n2} B & \cdots & a_{nm} B
    \end{array}
\right].
\]
\end{definition}
It is easy to show the following properties of the Kronecker product: let $A  \in \mathrm{Mat}_{n}(\mathbb{F}_p)$ and $B  \in \mathrm{Mat}_{m}(\mathbb{F}_p)$.  Then, 
\begin{itemize} 
    \item $(A \otimes B)^t = A^t \otimes B^t$ .
    \item $(AB) \otimes (CD) = (A \otimes C)(B \otimes D)$.
    \item If $A$ and $B$ are nonsingular, then so is $A \otimes B$ and $(A \otimes B)^{-1} = A^{-1} \otimes B^{-1}$.
\end{itemize}
For a proof of these facts see \cite[Sec. 4.2]{HornJohnson1991}.

\begin{lemma} Let $A  \in \mathrm{Mat}_{n}(\mathbb{F}_p)$ and $B \in \mathrm{Mat}_{m}(\mathbb{F}_p)$. Then $(A \otimes B)^\antidiagonaltranspose 
=
A^\antidiagonaltranspose \otimes B^\antidiagonaltranspose.
$
\end{lemma}

\begin{proof} Let $J_A$ and $J_B$ be the exchange matrices of sizes $n \times n$ and $m \times m$ respectively. Then it is clear that the matrix $J$ defined as $J = J_A \otimes J_B$ is the exchange matrix of size $nm \times nm$. Hence, by the properties of the Kronecker product, we have
\[ 
(A \otimes B)^\antidiagonaltranspose
=
J (A \otimes B)^t J
=
(J_A \otimes J_B) (A^t \otimes B^t) (J_A \otimes J_B)
=
(J_A A^t J_A) \otimes (J_B B^t J_B)
=
A^\antidiagonaltranspose \otimes B^\antidiagonaltranspose.
\]

\end{proof}

\begin{definition}
The \textbf{matrix exponential} $A$ is the matrix defined as
\[ e^A = \sum_{i=0}^\infty \frac{1}{i!} A^i. \]
\end{definition}

It holds $(e^A)^{-1} = e^{-A}$. For a proof of this fact, see \cite[6.2.38]{HornJohnson1991}. The next lemma shows that the matrix exponential behaves well with respect to the transposition with respect to the antidiagonal. This was already proved in \cite{AbramsFishkingValdesLeon2000}.

\begin{lemma} $e^{(A^\antidiagonaltranspose)} = \left( e^A \right)^\antidiagonaltranspose$.    
\end{lemma}

\begin{proof} Since $J^{-1} = J$, we have
\[  
(A^\antidiagonaltranspose)^i = (JA^tJ)^i = \underbrace{(JA^tJ) \cdots (JA^tJ)}_{\text{$i$ factors}} = J(\underbrace{A^t \cdots A^t}_{\text{$i$ factors}}) J = J (A^t)^i J = J (A^i)^t J = (A^i)^\antidiagonaltranspose. 
\]
Hence, 
\[ e^{(A^\antidiagonaltranspose)} 
= \sum_{i=0}^\infty \frac{1}{i!} (A^\antidiagonaltranspose)^i 
= \sum_{i=0}^\infty \frac{1}{i!} (J(A^i)^t J) 
= J \left( \sum_{i=0}^\infty \frac{1}{i!} A^i \right)^t J
= \left( \sum_{i=0}^\infty \frac{1}{i!} A^i \right)^\antidiagonaltranspose = \left( e^A \right)^\antidiagonaltranspose. \]
\end{proof}

Next, we introduce a matrix that will be important in the following sections.

\begin{definition}
The \textbf{Pascal-Hadamard matrix} of size $p^L \times p^L$ modulo $p$ is the matrix defined as
\begin{equation}\label{eq:Pascal_matrix}
H_L =
\left[ 
    \begin{array}{ccccc}
        \binom{0}{0} & \binom{1}{0} & \binom{2}{0} & \cdots & \binom{p^L-1}{0} \\[0.5em]
        \binom{0}{1} & \binom{1}{1} & \binom{2}{1} & \cdots & \binom{p^L-1}{1} \\[0.5em]
        \binom{0}{2} & \binom{1}{2} & \binom{2}{2} & \cdots & \binom{p^L-1}{2} \\
        \vdots & \vdots & \vdots & \ddots & \vdots \\
        \binom{0}{p^L-1} & \binom{1}{p^L-1} & \binom{2}{p^L-1} & \cdots & \binom{p^L-1}{p^L-1}
    \end{array}
\right].
\end{equation}    
\end{definition}

This matrix has been widely studied in the literature \cite{CallVellman1993,AbramsFishkingValdesLeon2000,Horn2012bk} and has several interesting properties.
It is used in various fields, such as dynamical systems theory, physics, engineering, computer science, and finance, among others, due to its ability to model processes that evolve continuously.

\begin{lemma}\label{lem:inverse_H} The Pascal-Hadamard matrix satisfies the following properties
\begin{enumerate}
    \item[a)]  Given the $p \times p$ matrix 
    \[ 
    U = 
    \left[
        \begin{array}{cccccccc}
            0 & 1 & 0 & \cdots & 0 & 0 & 0 \\
            0 & 0 & 2 & \cdots & 0 & 0 & 0 \\
            \vdots & \vdots & \vdots &  & \vdots & \vdots & \vdots \\
            0 & 0 & 0 & \cdots & p-3 & 0 & 0 \\
            0 & 0 & 0 & \cdots & 0 & p-2 & 0 \\
            0 & 0 & 0 & \cdots & 0 & 0 & p-1 \\
            0 & 0 & 0 & \cdots & 0 & 0 & 0
        \end{array}
    \right],
    \]
    it holds that $H_1 = e^U$.
    
    \item[b)] $H_L = H_1 \otimes H_{L-1}$.
\end{enumerate}
\end{lemma}

\begin{proof} A proof for a) can be found in \cite[Theorem 5]{CallVellman1993}. Property b) was already observed in \cite{Kubelka2002}.
\end{proof}

\begin{theorem}\label{thm:inversaHL} The inverse of $H_L$ is $H_L^{-1} = H_L^\antidiagonaltranspose$.    
\end{theorem}

\begin{proof} We proceed by induction over $L$:
\begin{itemize}
\item For $L=0$ it is trivial.

\item Let us prove the case $L=1$. Due to the recursive definition of $H_L$, we will need it to prove the general case. From Lemma~\ref{lem:inverse_H}a), $H_1 = e^U$. Notice that
\[ 
-U =
\left[ 
    \begin{array}{cccccccc}
        0 & p-1 & 0 & \cdots & 0 & 0 & 0 \\
        0 & 0 & p-2 & \cdots & 0 & 0 & 0 \\
        0 & 0 & 0 & \ddots & 0 & 0 & 0 \\
        \vdots & \vdots & \vdots & \ddots & \ddots & \vdots & \vdots \\
        0 & 0 & 0 & \cdots & 0 & 2 & 0 \\
        0 & 0 & 0 & \cdots & 0 & 0 & 1 \\
        0 & 0 & 0 & \cdots & 0 & 0 & 0
    \end{array}
\right]
= 
U^\antidiagonaltranspose,
\]
so we have that $\left( e^U \right)^{-1} = e^{-U}$. From here, we deduce that the inverse of $H_1$ is $H_1^\antidiagonaltranspose$ since
\[ 
H_1^{-1} = \left( e^U \right)^{-1} = e^{-U} = e^{\left( U^\antidiagonaltranspose \right)} = \left( e^U \right)^\antidiagonaltranspose = H_1^\antidiagonaltranspose.
\]

\item Assume that $H_L^{-1} = H_L^\antidiagonaltranspose$ for a certain $L$. Let us see that $H_{L+1}^{-1} = H_{L+1}^\antidiagonaltranspose$. By the properties of the Kronecker product, we have that
\[ H_{L+1}^{-1} = \left( H_1 \otimes H_L \right)^{-1} = H_1^{-1} \otimes H_L^{-1}. \]
By the induction hypothesis, and using that $H_1^{-1} = H_1^\antidiagonaltranspose$ and the properties of the antidiagonal transposition, we have
\[
H_{L+1}^{-1} = H_1^\antidiagonaltranspose \otimes H_L^\antidiagonaltranspose = \left( H_1 \otimes H_L \right)^\antidiagonaltranspose = H_{L+1}^\antidiagonaltranspose.
\]
\end{itemize}
\end{proof}

 \section{$p$-ary Binomial Sequences}

The binomial coefficient $\binom{n}{k}$ is the coefficient of the power $x^k$ in the polynomial expansion of $(1+x)^n$.
For every positive integer $n$, it is a well-known fact that $\binom{n}{0}=1$ and $\binom{n}{k}=0$ for $k>n$.
%When $n$ takes successive values, each binomial coefficient modulo 2 defines a binary sequence with constant period.
Moreover, it is worth noticing that if we arrange these binomial coefficients into rows for successive values of $n = 0, 1, 2, \ldots $, then the generated structure is the Pascal's triangle (see Figure~\ref{fig:pascal}).
The left-most diagonal is the identically 1 sequence, the next diagonal is the sequence of natural numbers $\{1, 2, 3, \ldots\}$,
the next one is the sequence of triangular numbers $\{1, 3, 6, 10, \ldots\}$, etc.
 % (numbers that count the objects that can form an equilateral triangle), etc.
Other fascinating sequences (tetrahedral numbers, pentatope numbers,
hexagonal numbers, Fibonacci sequence, etc.) can be found in the diagonals of this triangle.
Moreover, if we color the odd numbers of the Pascal's triangle and shade the even numbers, then we get the Sierpinski's triangle (see Figure~\ref{fig:sierp}).
This means that the Pascal's triangle modulo 2 generates the Sierpinski's triangle  (see Figure~\ref{fig:sierp2}).

The binomial coefficients reduced modulo $p$ allow us to introduce the concept of $p$-ary binomial sequence.
%In this section, the concept of binomial sequence is introduced as well as a fundamental result that relates binomial sequences with arbitrary binary sequences of period a power of $2$.

\begin{definition}
Given a fixed   integer $k\geq 0$,
the sequence $\{ s_n \}$ given by:
\begin{equation*}\label{def:bin}
s_n =
\begin{cases}
\phantom{(}0\phantom{)} &\text{ if } n<k \\
\binom{n}{k} \bmod p &\text{ if } n \geq k \\
\end{cases}
\end{equation*}
is known as the \textbf{$p$-ary $k$-th binomial sequence}.
\end{definition}
From now on, and if there is no confusion, we will refer to this sequence as the $k$-th binomial sequence and denote it by $\textbinomseq{n}{k}$. If the term $k$ is not determined, we simply call it a \textbf{binomial sequence}.

% \begin{table}[h]
% %\caption{Definition of the binomial coefficients and relation with the MSS-sequence in Example~\ref{ex:3} \label{tab:3}}
% \centering
% \caption{Binomial sequences, first eight terms, period and linear
% complexity\label{tab:1}}
% $
% \def\arraystretch{1.2}
% \begin{array}{|c|c|c|c|}\hline
% \text{Binomial coeff.}&\text{Binomial sequence} & \text{Period} & \text{LC}\\ \hline
% \binom{n}{0} & 1\ 1\ 1\ 1\ 1\ 1\ 1\ 1 & 1 & 1 \\
% \binom{n}{1} & 0\ 1\ 0\ 1\ 0\ 1\ 0\ 1 & 2 & 2 \\
% \binom{n}{2} & 0\ 0\ 1\ 1\ 0\ 0\ 1\ 1 & 4 & 3 \\
% \binom{n}{3} & 0\ 0\ 0\ 1\ 0\ 0\ 0\ 1 & 4 & 4 \\
% \binom{n}{4} & 0\ 0\ 0\ 0\ 1\ 1\ 1\ 1 & 8 & 5 \\
% \binom{n}{5} & 0\ 0\ 0\ 0\ 0\ 1\ 0\ 1 & 8 & 6 \\
% \binom{n}{6} & 0\ 0\ 0\ 0\ 0\ 0\ 1\ 1 & 8 & 7 \\
% \binom{n}{7} & 0\ 0\ 0\ 0\ 0\ 0\ 0\ 1 & 8 & 8 \\\hline
% \end{array}
% $
% \end{table}
% As an example, we can show Tables~\ref{tab:1} and \ref{tab:2}.
% Table~\ref{tab:1} shows the first 8 binary binomial sequences,  their periods   and their linear complexities. 
In \cite{Cardell2019a}, binary binomial sequences were deeply studied. It was showed that the diagonals of the Sierpinski triangle (Pascal's triangle modulo 2) correspond to shifted versions of these sequences (
see Figure~\ref{fig:sierp2}).
%, we can observe that 
% the diagonals of the Sierpinski triangle (Pascal's triangle modulo 2) 
% correspond to shifted versions of these sequences. %(see Figure~\ref{fig:sierp2}).
% Notice that the diagonals of the Sierpinski triangle (Pascal's triangle modulo 2) 
% correspond to shifted versions of these sequences (see Figure~\ref{fig:sierp2}).
Table~\ref{tab:2} shows the $27$ first $3$-ary binomial sequences, their corresponding periods, and linear complexities.
%The linear complexities of the binomial sequences are defined in Theorem~\ref{thm:LC} (Section~\ref{sec:prop}).
Recall that, in this case, the successive binomial sequences correspond to shifted versions of the successive diagonals in the Pascal's triangle reduced modulo $3$ (see Figure~\ref{fig:pascal3}).
Moreover, we can observe that the rows of the Pascal-Hadamard matrix given in expression~\eqref{eq:Pascal_matrix} are the %$p$-ary $k$-th 
binomial sequences.

\color{red}
 \begin{figure}
\caption{Binomial coefficients arranged as triangles}
\centering
\subfloat[Pascal's triangle \label{fig:pascal}]{
\centering
\begin{tikzpicture}[scale=1.75]
%\phantom{\draw (0,0) circle (1);}
\draw (0,0) node [scale=1.4] {$\binom{0}{0}$};

\draw (-0.3,-0.5) node[scale=1.4] {$\binom{1}{0}$};
\draw (0.3,-0.5) node[scale=1.4] {$\binom{1}{1}$};

\draw (-0.6,-1) node[scale=1.4] {$\binom{2}{0}$};
\draw (0,-1) node[scale=1.4] {$\binom{2}{1}$};
\draw (0.6,-1) node[scale=1.4] {$\binom{2}{2}$};

\draw (-0.9,-1.5) node[scale=1.4] {$\binom{3}{0}$};
\draw (-0.3,-1.5) node [scale=1.4]{$\binom{3}{1}$};
\draw (+0.3,-1.5) node[scale=1.4] {$\binom{3}{2}$};
\draw (0.9,-1.5) node [scale=1.4]{$\binom{3}{3}$};

\draw (-1.2,-2) node[scale=1.4] {$\binom{4}{0}$};
\draw (-0.6,-2) node[scale=1.4] {$\binom{4}{1}$};
\draw (-0,-2) node[scale=1.4] {$\binom{4}{2}$};
\draw (+0.6,-2) node[scale=1.4] {$\binom{4}{3}$};
\draw (1.2,-2) node[scale=1.4] {$\binom{4}{4}$};

\draw (-1.5,-2.5) node[scale=1.4] {$\binom{5}{0}$};
\draw (-0.9,-2.5) node[scale=1.4] {$\binom{5}{1}$};
\draw (-0.3,-2.5) node[scale=1.4] {$\binom{5}{2}$};
\draw (0.3,-2.5) node[scale=1.4] {$\binom{5}{3}$};
\draw (+0.9,-2.5) node[scale=1.4] {$\binom{5}{4}$};
\draw (1.5,-2.5) node[scale=1.4] {$\binom{5}{5}$};
\end{tikzpicture}
}
\qquad
\subfloat[Sierpinski's triangle\label{fig:sierp}]{
\centering
\begin{tikzpicture}[scale=0.7]
\phantom{\draw (0,0) circle (0.5);}

%------------------
 \draw [ color=black, fill=black!20](0,0)--(-0.5,-0.87)--(0.5,-0.87)--cycle;
\draw [color=black, fill=black!20](0.5,-0.87)--(0,-1.74)--(1,-1.74)--cycle;
\draw [color=black, fill=black!20](-0.5,-0.87)--(0,-1.74)--(-1,-1.74)--cycle;
\draw [color=black, fill=black!20](1,-1.74)--(1.5, -2.61)--(0.5,-2.61)--cycle;
\draw [color=black, fill=black!20](-1,-1.74)--(-1.5, -2.61)--(-0.5,-2.61)--cycle;
\draw [color=black, fill=black!20](1.5, -2.61)--(2,-3.48)--(1,-3.48)--cycle;
\draw [color=black, fill=black!20](0.5, -2.61)--(1,-3.48)--(0,-3.48)--cycle;
\draw [color=black, fill=black!20](-0.5, -2.61)--(-1,-3.48)--(0,-3.48)--cycle;
\draw [color=black, fill=black!20](-1.5, -2.61)--(-2,-3.48)--(-1,-3.48)--cycle;
%%%%%%%%%%%%%%%%%%%%%%%%%%%%%%%%%%%%%%%%%%%%%%%
\draw [ color=black, fill=black!20](-2,-3.48)--(-2.5,-4.35)--(-1.5,-4.35)--cycle;
\draw [color=black, fill=black!20](-1.5,-4.35)--(-2,-5.22)--(-1,-5.22)--cycle;
\draw [color=black, fill=black!20](-2.5,-4.35)--(-2,-5.22)--(-3,-5.22)--cycle;
\draw [color=black, fill=black!20](-1,-5.22)--(-0.5, -6.09)--(-1.5,-6.09)--cycle;
\draw [color=black, fill=black!20](-3,-5.22)--(-3.5, -6.09)--(-2.5,-6.09)--cycle;
\draw [color=black, fill=black!20](-0.5, -6.09)--(0,-6.96)--(-1,-6.96)--cycle;
\draw [color=black, fill=black!20](-1.5, -6.09)--(-1,-6.96)--(-2,-6.96)--cycle;
\draw [color=black, fill=black!20](-2.5, -6.09)--(-3,-6.96)--(-2,-6.96)--cycle;
\draw [color=black, fill=black!20](-3.5, -6.09)--(-4,-6.96)--(-3,-6.96)--cycle;
%%%%%%%%%%%%%%%%%%%%%%%%%%%%%%%%%%%%%%%%%%%%%%%%%%
%%%%%%%%%%%%%%%%%%%%%%%%%%%%%%%%%%%%%%%%%%%%%%%
\draw [ color=black, fill=black!20](2,-3.48)--(2.5,-4.35)--(1.5,-4.35)--cycle;
\draw [color=black, fill=black!20](1.5,-4.35)--(2,-5.22)--(1,-5.22)--cycle;
\draw [color=black, fill=black!20](2.5,-4.35)--(2,-5.22)--(3,-5.22)--cycle;
\draw [color=black, fill=black!20](1,-5.22)--(0.5, -6.09)--(1.5,-6.09)--cycle;
\draw [color=black, fill=black!20](3,-5.22)--(3.5, -6.09)--(2.5,-6.09)--cycle;
\draw [color=black, fill=black!20](0.5, -6.09)--(0,-6.96)--(1,-6.96)--cycle;
\draw [color=black, fill=black!20](1.5, -6.09)--(1,-6.96)--(2,-6.96)--cycle;
\draw [color=black, fill=black!20](2.5, -6.09)--(3,-6.96)--(2,-6.96)--cycle;
\draw [color=black, fill=black!20](3.5, -6.09)--(4,-6.96)--(3,-6.96)--cycle;
%%%%%%%%%%%%%%%%%%%%%%%%%%%%%%%%%%%%%%%%%%%%%%%nodos
\draw (0,-0.46) node [scale=0.65] {1};

\draw (0.5,-1.4) node [scale=0.65] {1};
\draw (-0.5,-1.4) node [scale=0.65] {1};

\draw (1,-2.23) node [scale=0.65] {1};\draw (0,-2.23) node [scale=0.65] {2};
\draw (-1,-2.23) node [scale=0.65] {1};

\draw (1.5,-3.12) node [scale=0.65] {1};\draw (0.5,-3.12) node [scale=0.65] {3};
\draw (-1.5,-3.12) node [scale=0.65] {1};\draw (-0.5,-3.12) node [scale=0.65] {3};

\draw (2,-4.01) node [scale=0.65] {1};\draw (1,-4.01) node [scale=0.65] {4};\draw (-1,-4.01) node [scale=0.65] {4};
\draw (-2,-4.01) node [scale=0.65] {1};\draw (0,-4.01) node [scale=0.65] {6};

\draw (2.5,-4.9) node [scale=0.65] {1};\draw (0.5,-4.9) node [scale=0.65] {10};
\draw (-2.5,-4.9) node [scale=0.65] {1};\draw (-0.5,-4.9) node [scale=0.65] {10};
\draw (1.5,-4.9) node [scale=0.65] {5};
\draw (-1.5,-4.9) node [scale=0.65] {5};

\draw (3,-5.79) node [scale=0.65] {1};\draw (-0,-5.79) node [scale=0.65] {20};
\draw (-3,-5.79) node [scale=0.65] {1};\draw (2,-5.79) node [scale=0.65] {6};\draw (-2,-5.79) node [scale=0.65] {6};
\draw (-1,-5.79) node [scale=0.65] {15};
\draw (1,-5.79) node [scale=0.65] {15};

\draw (3.5,-6.68) node [scale=0.65] {1};
\draw (-3.5,-6.68) node [scale=0.65] {1};
\draw (-2.5,-6.68) node [scale=0.65] {7};
\draw (-1.5,-6.68) node [scale=0.65] {21};
\draw (-0.5,-6.68) node [scale=0.65] {35};
\draw (0.5,-6.68) node [scale=0.65] {35};
\draw (1.5,-6.68) node [scale=0.65] {21};
\draw (2.5,-6.68) node [scale=0.65] {7};
\end{tikzpicture}
}
\qquad
\subfloat[Sierpinski's triangle mod 2\label{fig:sierp2}]{
\centering
\begin{tikzpicture}[scale=0.7]
\phantom{\draw (0,0) circle (0.5);}

%------------------
\draw [ color=black, fill=black!20](0,0)--(-0.5,-0.87)--(0.5,-0.87)--cycle;
\draw [color=black, fill=black!20](0.5,-0.87)--(0,-1.74)--(1,-1.74)--cycle;
\draw [color=black, fill=black!20](-0.5,-0.87)--(0,-1.74)--(-1,-1.74)--cycle;
\draw [color=black, fill=black!20](1,-1.74)--(1.5, -2.61)--(0.5,-2.61)--cycle;
\draw [color=black, fill=black!20](-1,-1.74)--(-1.5, -2.61)--(-0.5,-2.61)--cycle;
\draw [color=black, fill=black!20](1.5, -2.61)--(2,-3.48)--(1,-3.48)--cycle;
\draw [color=black, fill=black!20](0.5, -2.61)--(1,-3.48)--(0,-3.48)--cycle;
\draw [color=black, fill=black!20](-0.5, -2.61)--(-1,-3.48)--(0,-3.48)--cycle;
\draw [color=black, fill=black!20](-1.5, -2.61)--(-2,-3.48)--(-1,-3.48)--cycle;
%%%%%%%%%%%%%%%%%%%%%%%%%%%%%%%%%%%%%%%%%%%%%%%
\draw [ color=black, fill=black!20](-2,-3.48)--(-2.5,-4.35)--(-1.5,-4.35)--cycle;
\draw [color=black, fill=black!20](-1.5,-4.35)--(-2,-5.22)--(-1,-5.22)--cycle;
\draw [color=black, fill=black!20](-2.5,-4.35)--(-2,-5.22)--(-3,-5.22)--cycle;
\draw [color=black, fill=black!20](-1,-5.22)--(-0.5, -6.09)--(-1.5,-6.09)--cycle;
\draw [color=black, fill=black!20](-3,-5.22)--(-3.5, -6.09)--(-2.5,-6.09)--cycle;
\draw [color=black, fill=black!20](-0.5, -6.09)--(0,-6.96)--(-1,-6.96)--cycle;
\draw [color=black, fill=black!20](-1.5, -6.09)--(-1,-6.96)--(-2,-6.96)--cycle;
\draw [color=black, fill=black!20](-2.5, -6.09)--(-3,-6.96)--(-2,-6.96)--cycle;
\draw [color=black, fill=black!20](-3.5, -6.09)--(-4,-6.96)--(-3,-6.96)--cycle;
%%%%%%%%%%%%%%%%%%%%%%%%%%%%%%%%%%%%%%%%%%%%%%%%%%
%%%%%%%%%%%%%%%%%%%%%%%%%%%%%%%%%%%%%%%%%%%%%%%
\draw [ color=black, fill=black!20](2,-3.48)--(2.5,-4.35)--(1.5,-4.35)--cycle;
\draw [color=black, fill=black!20](1.5,-4.35)--(2,-5.22)--(1,-5.22)--cycle;
\draw [color=black, fill=black!20](2.5,-4.35)--(2,-5.22)--(3,-5.22)--cycle;
\draw [color=black, fill=black!20](1,-5.22)--(0.5, -6.09)--(1.5,-6.09)--cycle;
\draw [color=black, fill=black!20](3,-5.22)--(3.5, -6.09)--(2.5,-6.09)--cycle;
\draw [color=black, fill=black!20](0.5, -6.09)--(0,-6.96)--(1,-6.96)--cycle;
\draw [color=black, fill=black!20](1.5, -6.09)--(1,-6.96)--(2,-6.96)--cycle;
\draw [color=black, fill=black!20](2.5, -6.09)--(3,-6.96)--(2,-6.96)--cycle;
\draw [color=black, fill=black!20](3.5, -6.09)--(4,-6.96)--(3,-6.96)--cycle;
%%%%%%%%%%%%%%%%%%%%%%%%%%%%%%%%%%%%%%%%%%%%%%%nodos
\draw (0,-0.46) node [scale=0.65] {1};

\draw (0.5,-1.4) node [scale=0.65] {1};
\draw (-0.5,-1.4) node [scale=0.65] {1};

\draw (1,-2.23) node [scale=0.65] {1};
\draw (0,-2.23) node [scale=0.65] {0};
\draw (-1,-2.23) node [scale=0.65] {1};

\draw (1.5,-3.12) node [scale=0.65] {1};
\draw (0.5,-3.12) node [scale=0.65] {1};
\draw (-1.5,-3.12) node [scale=0.65] {1};
\draw (-0.5,-3.12) node [scale=0.65] {1};

\draw (2,-4.01) node [scale=0.65] {1}
;\draw (1,-4.01) node [scale=0.65] {0};\draw (-1,-4.01) node [scale=0.65] {0};
\draw (-2,-4.01) node [scale=0.65] {1};\draw (0,-4.01) node [scale=0.65] {0};

\draw (2.5,-4.9) node [scale=0.65] {1};
\draw (0.5,-4.9) node [scale=0.65] {0};
\draw (-2.5,-4.9) node [scale=0.65] {1};
\draw (-0.5,-4.9) node [scale=0.65] {0};
\draw (1.5,-4.9) node [scale=0.65] {1};
\draw (-1.5,-4.9) node [scale=0.65] {1};

\draw (3,-5.79) node [scale=0.65] {1};
\draw (-0,-5.79) node [scale=0.65] {0};
\draw (-3,-5.79) node [scale=0.65] {1};
\draw (2,-5.79) node [scale=0.65] {0};
\draw (-2,-5.79) node [scale=0.65] {0};
\draw (-1,-5.79) node [scale=0.65] {1};
\draw (1,-5.79) node [scale=0.65] {1};

\draw (3.5,-6.68) node [scale=0.65] {1};
\draw (-3.5,-6.68) node [scale=0.65] {1};
\draw (-2.5,-6.68) node [scale=0.65] {1};
\draw (-1.5,-6.68) node [scale=0.65] {1};
\draw (-0.5,-6.68) node [scale=0.65] {1};
\draw (0.5,-6.68) node [scale=0.65] {1};
\draw (1.5,-6.68) node [scale=0.65] {1};
\draw (2.5,-6.68) node [scale=0.65] {1};
\end{tikzpicture}}
\end{figure}
\color{black}

% Next result shows the relation between $p$-ary binomial sequences and $p$-ary sequences with period a power of $p$.

% \begin{theorem}[\cite{Cardell2019a}]
% Let $\{z_n\}$ be a $p$-ary sequence with period $T = p^L$, $L$ being a positive integer. Then, every $p$-ary sequence $\{s_n\}$ can be written as a linear combination of binomial sequences.
% \end{theorem}
%  \begin{proof}
% Since the period of $\{z_n\}$ is a power of $p$, then the next equation holds:
 
% \begin{equation}\label{eq:5}
% (E^{p^{L}}- 1) z_n =(E - 1)^{p^{L}} z_n=0,
% \end{equation}
% due to the characteristic of $\mathbb{F}_p$ is $p$. Moreover, the characteristic polynomial $p_m(x)=p(x)^m=(x-1)^m$. Therefore, its binary solutions are given by the equation~(\ref{eq:4}), which has now the following simplified form:
% \[
% z_n=\binom{n}{0}\,c_0 + \binom{n}{1}\,c_1 + \cdots + \binom{n}{T-1}\,c_{T-1} \quad \text{for }n\geq 0,
% \]
% where
% $1$ is the unique root of the polynomial $(x+1)^T$ with multiplicity $T = 2^L$, the coefficients $c_{i}\in\Fset_2$ and $\binom{n}{i}$ are binomial coefficients modulo 2.
%  
%  \end{proof}
% Different choices of $ c_{i}$ will produce different sequences $\{z_n\}$ with distinct characteristics and properties, but all of them with period $2^l$, $0 \leq l\leq L$.

% Binary binomial sequences were deeply studied in \cite{Cardell2019a}.
In this work we study the binomial sequences over $\mathbb{F}_p$, with $p>2$ a prime.
In the following subsections, we analyse the cases where $p \in \{3,5\}$ in order to illustrate the idea of the general case. The proofs of the results for these particular cases are omitted since we will give the proofs for the general case in Subsection~\ref{Subsect:F_p}.  
 
%%%%%%%%%%%%%%%%%%%%%%%%
%%%%%%%%%%%%%%%%%%%%
%\section{Main results}
\subsection{Sequences over $\mathbb{F}_3$ and $\mathbb{F}_5$} \label{sub:seq3_5}

In this section, we study the $3$-ary and $5$-ary binomial sequences over $\mathbb{F}_3$ and  $\mathbb{F}_5$; we observe some properties that will help us to generalize the case over $\mathbb{F}_p$.

% \begin{theorem}
%     The period of the sequence $\binom{n}{k}$, where $k=3^t+r$ and $r<2\cdot 3^t$ is $T=3^{t+1}$.
%   Furthermore, the sequences have the form   
%     $$\binom{n}{3^t+k}=\left\{\underbrace{00\ldots00}_{3^{t}}, \binom{n}{k} ,2\binom{n}{k}\right\}$$
%     when $k<3^t$ and the form
%     $$\binom{n}{3^t+k}=\left\{\underbrace{00\ldots00}_{2\cdot 3^{t}}, \binom{n}{k} \right\}$$
%     when $k>3^t$.  

% \end{theorem}

In Table~\ref{tab:2}, we observe the first 27 binomial sequences over $\mathbb{F}_3$. A recurrence in the formation of the sequences can be observed.

Let $k$ be a positive integer and $L, r$ be non-negative integers. %For $j \in \{1,2\}$, consider $\binomseq{n}{j}\right\}=(s^{(j)}_0,s^{(j)}_1, s^{(j)}_{2})$. 
The period of the sequence $\textbinomseq{n}{k}$, with $k=3^L+r$ and $0 \leq r<2\cdot 3^L$, is $T=3^{L+1}$.
  Furthermore, the sequences have the form:
  \vspace{0.6em}
  \begin{enumerate}[(a),topsep=0pt,partopsep=0pt,parsep=0pt,itemsep=0pt]
\item For the multiples of the power of $3$, we have that:
  \begin{align}
  \binomseq{n}{3^L} & =\left\{\underbrace{00\ldots00}_{3^{L}}, \underbrace{11\ldots11}_{3^{L}} ,\underbrace{22\ldots22}_{3^{L}}\right\},
  % = \left\{\underbrace{s^{(1)}_0\ldots s^{(1)}_0}_{3^{L}}, \underbrace{s^{(1)}_1\ldots s^{(1)}_1}_{3^{L}}, \ldots, \underbrace{s^{(1)}_{2}\ldots s^{(1)}_{2}}_{3^{L}}\right\}
  \label{seq3.1} \\
    % \binom{n}{3^L+r}& =\left\{\underbrace{00\ldots00}_{3^{L}}, \underbrace{\binom{n}{r}}_{3^{L-i}} ,\underbrace{2\binom{n}{r}}_{3^{L-i}}\right\}
    % \text{ if } 3^{i-1}\leq r<3^{i} \text{ for } i=1,2,\ldots, L. \label{seq3.2}\\
    \binomseq{n}{2\cdot 3^L} & =\left\{\underbrace{00\ldots00}_{2 \cdot 3^{L}}, %\underbrace{00\ldots00}_{3^{L}}
    \underbrace{11\ldots11}_{3^{L}}
    \right\}.
    % = \left\{\underbrace{s^{(2)}_0\ldots s^{(2)}_0}_{3^{L}}, \underbrace{s^{(2)}_1\ldots s^{(2)}_1}_{3^{L}}, \ldots, \underbrace{s^{(2)}_{2}\ldots s^{(2)}_{2}}_{3^{L}}\right\}
    \label{seq3.3} 
    % \binom{n}{2\cdot 3^L+r}& =\left\{\underbrace{00\ldots00}_{2\cdot 3^{L}}, \underbrace{\binom{n}{r}}_{3^{L-i}}\right\}
    % \text{ if }3^{i-1}\leq r<3^{i} \text{ for } i=1,2,\ldots, L. \label{seq3.4} 
   \end{align}
   \item For $i=1,2,\ldots, L$ and  $3^{i-1}\leq r<3^{i}$, we have that
  \begin{align}
  % \binom{n}{3^L}& =\left\{\underbrace{00\ldots00}_{3^{L}}, \underbrace{11\ldots11}_{3^{L}} ,\underbrace{22\ldots22}_{3^{L}}\right\} \label{seq3.1} \\
   \binomseq{n}{3^L+r} & =\left\{\underbrace{00\ldots00}_{3^{L}}, \underbrace{\binomseq{n}{r}}_{3^{L-i}} ,\underbrace{2\binomseq{n}{r}}_{3^{L-i}}\right\},
   % \text{ if } 3^{i-1}\leq r<3^{i} \text{ for } i=1,2,\ldots, L. 
   \label{seq3.2}
   \\
    % \binom{n}{2\cdot 3^L}& =\left\{\underbrace{00\ldots00}_{2\cdot 3^{L}}, \underbrace{11\ldots11}_{3^{L}}\right\} \label{seq3.3} \\
     \binomseq{n}{2\cdot 3^L+r} & =\bigg\{\underbrace{00\ldots00}_{2\cdot 3^{L}}, \underbrace{\binomseq{n}{r}}_{3^{L-i}}\bigg\}.
   % \text{ if }3^{i-1}\leq r<3^{i} \text{ for } i=1,2,\ldots, L. 
   \label{seq3.4} 
   \end{align}
   where $\underbrace{\binomseq{n}{r}}_{3^{L-i}}$ means that the sequence $\textbinomseq{n}{r}$ is repeated $3^{L-i}$ times.
\end{enumerate}
Observe that    the sequence $\textbinomseq{n}{3^L-1}$ with $L\geq 1$, is of the form $\big\{\underbrace{00\ldots00}_{3^L-1} 1\big\}$ with period $3^L$.

Next example illustrates how sequences are recursively constructed from each other.

\begin{example}
Consider the $3$-ary sequence $\textbinomseq{n}{10}$ given in Table~\ref{tab:2}. We have that $k=10=3^2+1$ with $r=1$ satisfying that $1 \leq r < 3$. 
%From expression~\eqref{seq3.2}, 
We have that
\[
\binomseq{n}{3^2+1} = 
\begin{array}{ccccccccc|ccccccccc|ccccccccc}
\left\{\textcolor{blue}{000000000}\textcolor{red}{012012012}\textcolor{green}{021021021}\right\}
\end{array}
=
\left\{
\underbrace{\textcolor{blue}{000000000}}_{3^{2}},
\underbrace{\textcolor{red}{\binomseq{n}{1}}}_{3^1},
\underbrace{\textcolor{green}{2\binomseq{n}{1}}}_{3^1}
\right\}.
\]
% where the sequences $\textbinomseq{n}{1}$ and $\textbinomseq{n}{10}$ appear in Table~\ref{tab:2}.

Now, if we consider the $3$-ary sequence $\textbinomseq{n}{12}$, we have that $k=10=3^2+3$ with $r=3$ satisfying that $3 \leq r < 3^2$. 
%From expression~\ref{seq3.2}, 
Now, we obtain
\[
\binomseq{n}{3^2+3}=
\begin{array}{ccccccccc|ccccccccc|ccccccccc}
\left\{\textcolor{blue}{000000000}\textcolor{red}{000111222}\textcolor{green}{000222111}\right\}
=
\left\{\underbrace{\textcolor{blue}{000000000}}_{3^{2}}, \underbrace{\textcolor{red}{\binomseq{n}{3}}}_{3^0} ,\underbrace{\textcolor{green}{2\binomseq{n}{3}}}_{3^0}\right\}.
\end{array}
\]
% where the sequences $\textbinomseq{n}{3}$ and $\textbinomseq{n}{12}$ appear in Table~\ref{tab:2}.

Finally, we take the sequence $\textbinomseq{n}{22}$. We have that $k=22=2\cdot3^2+4$ with $r=4$ satisfying that $3 \leq r < 3^2$. 
We have that
\[
\binomseq{n}{2\cdot3^2+4}=
\begin{array}{cccccccccccccccccc|ccccccccc}
\left\{\textcolor{blue}{000000000000000000}\textcolor{red}{000012021}\right\}
=
\left\{\underbrace{\textcolor{blue}{000000000000000000}}_{2 \cdot 3^{2}}, \underbrace{\textcolor{red}{\binomseq{n}{4}}}_{3^0}\right\},
\end{array}
\]
which can be checked in Table~\ref{tab:2}.

\end{example}

Through the previous example, we observe that the length of the sequences defined by $\textbinomseq{n}{r}$ varies depending on the range in which $r$ falls.
We can deduce the following construction for $3$-ary sequences, which will be generalised and proved in this paper.

It is worth noticing that these sequences are the diagonals of the Pascal's triangle modulo 3. In Figure~\ref{fig:pascal3}, we find this triangle. It is possible to see the recurrence properties already mentioned and the fractal form of the triangle.
\begin{figure}
    \centering
\input{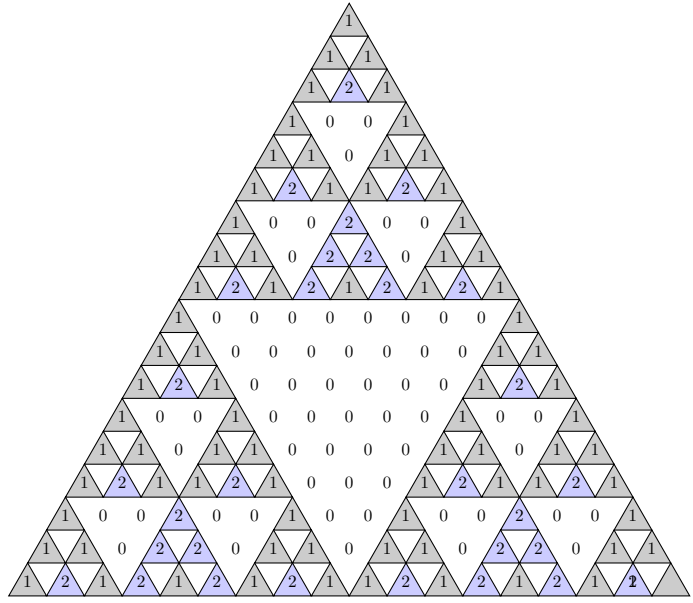}    \caption{Pascal's triangle modulo 3}
    \label{fig:pascal3}
\end{figure}

\begin{table}
\caption{First 27 binomial sequences over $\mathbb{F}_3$, periods   and linear complexities. }\label{tab:2}
\def\arraystretch{1.1}
\begin{tabular}{|c|c|c|c|}\hline
\text{Binomial sequence}&\text{First terms} & \text{Period}& \text{LC}\\ \hline
$\textbinomseq{n}{0}$  & \textcolor{   red}{1\:\!1\:\!1\:\!1\:\!1\:\!1\:\!1\:\!1\:\!1}\:\!1\:\!1\:\!1\:\!1\:\!1\:\!1\:\!1\:\!1\:\!1\:\!1\:\!1\:\!1\:\!1\:\!1\:\!1\:\!1\:\!1\:\!1 &  1 &  1 \\
$\textbinomseq{n}{1}$  & \textcolor{  blue}{0\:\!1\:\!2\:\!0\:\!1\:\!2\:\!0\:\!1\:\!2}\:\!0\:\!1\:\!2\:\!0\:\!1\:\!2\:\!0\:\!1\:\!2\:\!0\:\!1\:\!2\:\!0\:\!1\:\!2\:\!0\:\!1\:\!2 &  3 &  2 \\
$\textbinomseq{n}{2} $ & \textcolor{orange}{0\:\!0\:\!1\:\!0\:\!0\:\!1\:\!0\:\!0\:\!1}\:\!0\:\!0\:\!1\:\!0\:\!0\:\!1\:\!0\:\!0\:\!1\:\!0\:\!0\:\!1\:\!0\:\!0\:\!1\:\!0\:\!0\:\!1 &  3 &  3 \\
$\textbinomseq{n}{3}$  & \textcolor{purple}{0\:\!0\:\!0\:\!1\:\!1\:\!1\:\!2\:\!2\:\!2}\:\!0\:\!0\:\!0\:\!1\:\!1\:\!1\:\!2\:\!2\:\!2\:\!0\:\!0\:\!0\:\!1\:\!1\:\!1\:\!2\:\!2\:\!2 &  9 &  4 \\
$\textbinomseq{n}{4}$  & \textcolor{yellow}{0\:\!0\:\!0\:\!0\:\!1\:\!2\:\!0\:\!2\:\!1}\:\!0\:\!0\:\!0\:\!0\:\!1\:\!2\:\!0\:\!2\:\!1\:\!0\:\!0\:\!0\:\!0\:\!1\:\!2\:\!0\:\!2\:\!1 &  9 &  5 \\
$\textbinomseq{n}{5}$  & \textcolor{  pink}{0\:\!0\:\!0\:\!0\:\!0\:\!1\:\!0\:\!0\:\!2}\:\!0\:\!0\:\!0\:\!0\:\!0\:\!1\:\!0\:\!0\:\!2\:\!0\:\!0\:\!0\:\!0\:\!0\:\!1\:\!0\:\!0\:\!2 &  9 &  6 \\
$\textbinomseq{n}{6}$  & \textcolor{orange}{0\:\!0\:\!0\:\!0\:\!0\:\!0\:\!1\:\!1\:\!1}\:\!0\:\!0\:\!0\:\!0\:\!0\:\!0\:\!1\:\!1\:\!1\:\!0\:\!0\:\!0\:\!0\:\!0\:\!0\:\!1\:\!1\:\!1 &  9 &  7 \\
$\textbinomseq{n}{7}$  & \textcolor{  gray}{0\:\!0\:\!0\:\!0\:\!0\:\!0\:\!0\:\!1\:\!2}\:\!0\:\!0\:\!0\:\!0\:\!0\:\!0\:\!0\:\!1\:\!2\:\!0\:\!0\:\!0\:\!0\:\!0\:\!0\:\!0\:\!1\:\!2 &  9 &  8 \\
$\textbinomseq{n}{8}$  & \textcolor{ green}{0\:\!0\:\!0\:\!0\:\!0\:\!0\:\!0\:\!0\:\!1}\:\!0\:\!0\:\!0\:\!0\:\!0\:\!0\:\!0\:\!0\:\!1\:\!0\:\!0\:\!0\:\!0\:\!0\:\!0\:\!0\:\!0\:\!1 &  9 &  9 \\
$\textbinomseq{n}{9}$  & 0\:\!0\:\!0\:\!0\:\!0\:\!0\:\!0\:\!0\:\!0\:\!\textcolor{   red}{1\:\!1\:\!1\:\!1\:\!1\:\!1\:\!1\:\!1\:\!1}\:\!2\:\!2\:\!2\:\!2\:\!2\:\!2\:\!2\:\!2\:\!2 & 27 & 10 \\
$\textbinomseq{n}{10}$ & 0\:\!0\:\!0\:\!0\:\!0\:\!0\:\!0\:\!0\:\!0\:\!\textcolor{  blue}{0\:\!1\:\!2\:\!0\:\!1\:\!2\:\!0\:\!1\:\!2}\:\!0\:\!2\:\!1\:\!0\:\!2\:\!1\:\!0\:\!2\:\!1 & 27 & 11 \\
$\textbinomseq{n}{11}$ & 0\:\!0\:\!0\:\!0\:\!0\:\!0\:\!0\:\!0\:\!0\:\!\textcolor{orange}{0\:\!0\:\!1\:\!0\:\!0\:\!1\:\!0\:\!0\:\!1}\:\!0\:\!0\:\!2\:\!0\:\!0\:\!2\:\!0\:\!0\:\!2 & 27 & 12 \\
$\textbinomseq{n}{12}$ & 0\:\!0\:\!0\:\!0\:\!0\:\!0\:\!0\:\!0\:\!0\:\!\textcolor{purple}{0\:\!0\:\!0\:\!1\:\!1\:\!1\:\!2\:\!2\:\!2}\:\!0\:\!0\:\!0\:\!2\:\!2\:\!2\:\!1\:\!1\:\!1 & 27 & 13 \\
$\textbinomseq{n}{13}$ & 0\:\!0\:\!0\:\!0\:\!0\:\!0\:\!0\:\!0\:\!0\:\!\textcolor{yellow}{0\:\!0\:\!0\:\!0\:\!1\:\!2\:\!0\:\!2\:\!1}\:\!0\:\!0\:\!0\:\!0\:\!2\:\!1\:\!0\:\!1\:\!2 & 27 & 14 \\
$\textbinomseq{n}{14}$ & 0\:\!0\:\!0\:\!0\:\!0\:\!0\:\!0\:\!0\:\!0\:\!\textcolor{  pink}{0\:\!0\:\!0\:\!0\:\!0\:\!1\:\!0\:\!0\:\!2}\:\!0\:\!0\:\!0\:\!0\:\!0\:\!2\:\!0\:\!0\:\!1 & 27 & 15 \\
$\textbinomseq{n}{15}$ & 0\:\!0\:\!0\:\!0\:\!0\:\!0\:\!0\:\!0\:\!0\:\!\textcolor{orange}{0\:\!0\:\!0\:\!0\:\!0\:\!0\:\!1\:\!1\:\!1}\:\!0\:\!0\:\!0\:\!0\:\!0\:\!0\:\!2\:\!2\:\!2 & 27 & 16 \\
$\textbinomseq{n}{16}$ & 0\:\!0\:\!0\:\!0\:\!0\:\!0\:\!0\:\!0\:\!0\:\!\textcolor{  gray}{0\:\!0\:\!0\:\!0\:\!0\:\!0\:\!0\:\!1\:\!2}\:\!0\:\!0\:\!0\:\!0\:\!0\:\!0\:\!0\:\!2\:\!1 & 27 & 17 \\
$\textbinomseq{n}{17}$ & 0\:\!0\:\!0\:\!0\:\!0\:\!0\:\!0\:\!0\:\!0\:\!\textcolor{ green}{0\:\!0\:\!0\:\!0\:\!0\:\!0\:\!0\:\!0\:\!1}\:\!0\:\!0\:\!0\:\!0\:\!0\:\!0\:\!0\:\!0\:\!2 & 27 & 18 \\
$\textbinomseq{n}{18}$ & 0\:\!0\:\!0\:\!0\:\!0\:\!0\:\!0\:\!0\:\!0\:\!0\:\!0\:\!0\:\!0\:\!0\:\!0\:\!0\:\!0\:\!0\:\!\textcolor{   red}{1\:\!1\:\!1\:\!1\:\!1\:\!1\:\!1\:\!1\:\!1} & 27 & 19 \\
$\textbinomseq{n}{19}$ & 0\:\!0\:\!0\:\!0\:\!0\:\!0\:\!0\:\!0\:\!0\:\!0\:\!0\:\!0\:\!0\:\!0\:\!0\:\!0\:\!0\:\!0\:\!\textcolor{  blue}{0\:\!1\:\!2\:\!0\:\!1\:\!2\:\!0\:\!1\:\!2} & 27 & 20 \\
$\textbinomseq{n}{20}$ & 0\:\!0\:\!0\:\!0\:\!0\:\!0\:\!0\:\!0\:\!0\:\!0\:\!0\:\!0\:\!0\:\!0\:\!0\:\!0\:\!0\:\!0\:\!\textcolor{orange}{0\:\!0\:\!1\:\!0\:\!0\:\!1\:\!0\:\!0\:\!1} & 27 & 21 \\
$\textbinomseq{n}{21}$ & 0\:\!0\:\!0\:\!0\:\!0\:\!0\:\!0\:\!0\:\!0\:\!0\:\!0\:\!0\:\!0\:\!0\:\!0\:\!0\:\!0\:\!0\:\!\textcolor{purple}{0\:\!0\:\!0\:\!1\:\!1\:\!1\:\!2\:\!2\:\!2} & 27 & 22 \\
$\textbinomseq{n}{22}$ & 0\:\!0\:\!0\:\!0\:\!0\:\!0\:\!0\:\!0\:\!0\:\!0\:\!0\:\!0\:\!0\:\!0\:\!0\:\!0\:\!0\:\!0\:\!\textcolor{yellow}{0\:\!0\:\!0\:\!0\:\!1\:\!2\:\!0\:\!2\:\!1} & 27 & 23 \\
$\textbinomseq{n}{23}$ & 0\:\!0\:\!0\:\!0\:\!0\:\!0\:\!0\:\!0\:\!0\:\!0\:\!0\:\!0\:\!0\:\!0\:\!0\:\!0\:\!0\:\!0\:\!\textcolor{  pink}{0\:\!0\:\!0\:\!0\:\!0\:\!1\:\!0\:\!0\:\!2} & 27 & 24 \\
$\textbinomseq{n}{24}$ & 0\:\!0\:\!0\:\!0\:\!0\:\!0\:\!0\:\!0\:\!0\:\!0\:\!0\:\!0\:\!0\:\!0\:\!0\:\!0\:\!0\:\!0\:\!\textcolor{orange}{0\:\!0\:\!0\:\!0\:\!0\:\!0\:\!1\:\!1\:\!1} & 27 & 25 \\
$\textbinomseq{n}{25}$ & 0\:\!0\:\!0\:\!0\:\!0\:\!0\:\!0\:\!0\:\!0\:\!0\:\!0\:\!0\:\!0\:\!0\:\!0\:\!0\:\!0\:\!0\:\!\textcolor{  gray}{0\:\!0\:\!0\:\!0\:\!0\:\!0\:\!0\:\!1\:\!2} & 27 & 26 \\
$\textbinomseq{n}{26}$ & 0\:\!0\:\!0\:\!0\:\!0\:\!0\:\!0\:\!0\:\!0\:\!0\:\!0\:\!0\:\!0\:\!0\:\!0\:\!0\:\!0\:\!0\:\!\textcolor{ green}{0\:\!0\:\!0\:\!0\:\!0\:\!0\:\!0\:\!0\:\!1} & 27 & 27 \\
\hline
\end{tabular}
\end{table}

% \begin{itemize}
%     \item The period of $\binom{n}{k}$, where $k=3^t+r$, is $T=3^{t+1}$.
%     \item 
%     $\binom{n}{3^t}=\underbrace{00\ldots00}_{3^t} \underbrace{11\ldots11}_{3^t}\underbrace{22\ldots 22}_{3^t}$
%         \item 
%     $\binom{n}{3^t-1}=\underbrace{00\ldots00}_{3^t-1} 1$

%      \item 
%     $\binom{n}{3^t+k}=[\underbrace{00\ldots00}_{3^{t}}, \binom{n}{k} ,2\binom{n}{k}], k<3^t$

%        \item 
%     $\binom{n}{3^t+k}=[\underbrace{00\ldots00}_{2\cdot 3^{t}}, \binom{n}{k} ], k>3^t$
% \end{itemize}

%\subsection{Sequences over $\mathbb{F}_5$}

%As in the previous subsection,
With the aim of obtaining a generalisation over $\mathbb{F}_p$, we can provide a construction for $5$-ary binomial sequences over $\mathbb{F}_5$. These sequences appear in Appendix~A.
%with the aim of obtaining a generalisation over $\mathbb{F}_p$.

% \begin{table}[h]
% \caption{First $5$ binomial sequences over $\mathbb{F}_5$, first $5$ terms, periods and linear complexities. }\label{tab:p5}
% \centering
%  $
% \def\arraystretch{1.1}
% \begin{array}{|c|c|c|c|}\hline
% \text{Binomial sequences}&\text{First terms} & \text{Period}& \text{LC}\\ \hline
% \textbinomseq{n}{0} & \textcolor{   red}{1\:\!1\:\!1\:\!1\:\!1} &  1 &  1 \\
% \textbinomseq{n}{1} & \textcolor{  blue}{0\:\!1\:\!2\:\!3\:\!4} &  5 &  2 \\
% \textbinomseq{n}{2} & \textcolor{orange}{0\:\!0\:\!1\:\!3\:\!1}&  5 &  3 \\
% \textbinomseq{n}{3} & \textcolor{purple}{0\:\!0\:\!0\:\!1\:\!4}&  5 &  4 \\
% \textbinomseq{n}{4} & \textcolor{green}{0\:\!0\:\!0\:\!0\:\!1}&  5 &  5 \\
% \hline
% \end{array}
% $
% \end{table}

\begin{table}
\caption{First $5$ binomial sequences over $\mathbb{F}_5$, first $5$ terms, periods and linear complexities. }\label{tab:p5}
\centering
\def\arraystretch{1.1}
\begin{tabular}{|c|c|c|c|}\hline
\text{Binomial sequences}&\text{First terms} & \text{Period}& \text{LC}\\ \hline
$\textbinomseq{n}{0}$ &$ \textcolor{   red}{1\:\!1\:\!1\:\!1\:\!1}$ &  1 &  1 \\
$\textbinomseq{n}{1}$ & $\textcolor{  blue}{0\:\!1\:\!2\:\!3\:\!4}$ &  5 &  2 \\
$\textbinomseq{n}{2}$ & $\textcolor{orange}{0\:\!0\:\!1\:\!3\:\!1}$&  5 &  3 \\
$\textbinomseq{n}{3}$ & $\textcolor{purple}{0\:\!0\:\!0\:\!1\:\!4}$&  5 &  4 \\
$\textbinomseq{n}{4}$ & $\textcolor{green}{0\:\!0\:\!0\:\!0\:\!1}$&  5 &  5 \\
\hline
\end{tabular}
\end{table}

Let $k$ be a positive integer and $L, r$ be non-negative integers. The period of the sequence $\textbinomseq{n}{k}$, where $k=5^L+r$ and $0 \leq r<4\cdot 5^L$
    is $T=5^{L+1}$.
  Furthermore, the sequences have the following forms:
  \vspace{0.6em}
  \begin{enumerate}[(a),topsep=0pt,partopsep=0pt,parsep=0pt,itemsep=0pt]
      \item For the multiples of the power of $5$, we have that:
       \begin{align*}
   \binomseq{n}{\textcolor{blue}{1}\cdot5^L} & =\color{blue}\left\{\underbrace{00\ldots00}_{ 5^{L}}, \underbrace{11\ldots11}_{5^{L}},\underbrace{22\ldots22}_{5^{L}}, \underbrace{33\ldots33}_{5^{L}}, \underbrace{44\ldots44}_{5^{L}}\right\}, \\
  \binomseq{n}{\textcolor{orange}{2} \cdot 5^L}& =\left\{\color{orange}\underbrace{00\ldots00}_{5^{L}}, \underbrace{00\ldots00}_{5^{L}}, \underbrace{11\ldots11}_{5^{L}},\underbrace{33\ldots33}_{5^{L}}, \underbrace{11\ldots11}_{5^{L}}\color{black}\right\}, \\
   \binomseq{n}{\textcolor{purple}{3} \cdot 5^L}& =\left\{\color{purple}\underbrace{00\ldots00}_{5^{L}}, \underbrace{00\ldots00}_{5^{L}}, \underbrace{00\ldots00}_{5^{L}} \underbrace{11\ldots11}_{5^{L}},\underbrace{44\ldots44}_{5^{L}}\color{black}\right\},  \\
   \binomseq{n}{\textcolor{green}{4} \cdot 5^L}& =\left\{\color{green}\underbrace{00\ldots00}_{5^{L}}, \underbrace{00\ldots00}_{5^{L}},\underbrace{00\ldots00}_{5^{L}},\underbrace{00\ldots00}_{5^{L}},\underbrace{11\ldots11}_{5^{L}} \color{black}\right\}.
  \end{align*}
  % \begin{align}
  %  \binomseq{n}{5^L}& =\left\{\underbrace{00\ldots00}_{5^{L}}, \underbrace{11\ldots11}_{5^{L}},\underbrace{22\ldots22}_{5^{L}}, \underbrace{33\ldots33}_{5^{L}}, \underbrace{44\ldots44}_{5^{L}}\right\} \label{seqa5.1} \\
  % \binomseq{n}{2 \cdot 5^L}\right\}& =\left\{\underbrace{00\ldots00}_{2 \cdot 5^{L}}, \underbrace{11\ldots11}_{5^{L}},\underbrace{33\ldots33}_{5^{L}}, \underbrace{11\ldots11}_{5^{L}}\right\} \label{seqa5.2} \\
  %  \binomseq{n}{3 \cdot 5^L}\right\}& =\left\{\underbrace{00\ldots00}_{3 \cdot 5^{L}}, \underbrace{11\ldots11}_{5^{L}},\underbrace{44\ldots44}_{5^{L}}\right\} \label{seqa5.3} \\
  %  \binomseq{n}{4 \cdot 5^L}& =\left\{\underbrace{00\ldots00}_{4 \cdot 5^{L}}, \underbrace{11\ldots11}_{5^{L}}\right\} \label{seqa5.4}
  % \end{align}
  \item 
  % For $i=1,2, \ldots L$ and $5^{i-1}\leq r<5^{i}$, we have that
  % \begin{align}
  %   \binom{n}{5^L+r}& =\left\{\underbrace{00\ldots00}_{5^{L}}, \underbrace{\binom{n}{r}}_{5^{L-i}} ,\underbrace{2\binom{n}{r}}_{5^{L-i}},\underbrace{3\binom{n}{r}}_{5^{L-i}}, \underbrace{4\binom{n}{r}}_{5^{L-i}}\right\}
  %   %\text{ if } \ 5^{i-1}\leq r<5^{i}  \text{ for } i=1,2, \ldots L. 
  %   \label{seqb5.1}\\
  %   \binom{n}{2\cdot 5^L+r}& =\left\{\underbrace{00\ldots00}_{2\cdot 5^{L}}, \underbrace{\binom{n}{r}}_{5^{L-i}}, \underbrace{3\binom{n}{r}}_{5^{L-i}}, \underbrace{\binom{n}{r}}_{5^{L-i}}\right\}   %\text{ if } \ 5^{i-1}\leq r<5^{i}  \text{ for } i=1,2, \ldots L. 
  %   \label{seqb5.2} \\
  %   \binom{n}{3\cdot 5^L+r}& =\left\{\underbrace{00\ldots00}_{3\cdot 5^{L}}, \underbrace{\binom{n}{r}}_{5^{L-i}},\underbrace{4\binom{n}{r}}_{5^{L-i}}\right\}
  %  % \text{ if } \ 5^{i-1}\leq r<5^{i}  \text{ for } i=1,2, \ldots L. 
  %  \label{seqb5.3} \\
  %   \binom{n}{4 \cdot 5^L+r}& =\left\{\underbrace{00\ldots00}_{4\cdot 5^{L}}, \underbrace{\binom{n}{r}}_{5^{L-i}}\right\}
  %  %\text{ if } \ 5^{i-1}\leq r<5^{i}  \text{ for } i=1,2, \ldots L. 
  %  \label{seqb5.4}
  For $i=1,2, \ldots L$ and $5^{i-1}\leq r<5^{i}$,
  \begin{align*}
  \binomseq{n}{\textcolor{blue}{1}\cdot 5^L+r} 
  & =\left\{
  \underbrace{\textcolor{blue}{0} \cdot\binomseq{n}{r}}_{5^{L-i}}, \underbrace{\textcolor{blue}{1}\cdot\binomseq{n}{r}}_{5^{L-i}} ,\underbrace{\textcolor{blue}{2}\cdot\binomseq{n}{r}}_{5^{L-i}},\underbrace{\textcolor{blue}{3}\cdot\binomseq{n}{r}}_{5^{L-i}}, \underbrace{\textcolor{blue}{4}\cdot\binomseq{n}{r}}_{5^{L-i}}\right\}, \\
   % \text{ if } \ 5^{i-1}\leq r<5^{i}  \text{ for } i=1,2, \ldots L\\
   \binomseq{n}{\textcolor{orange}{2}\cdot 5^L+r} & =
    \left\{
        \underbrace{\textcolor{orange}{0} \cdot \binomseq{n}{r}}_{5^{L-i}},
        \underbrace{\textcolor{orange}{0} \cdot \binomseq{n}{r}}_{5^{L-i}},
        \underbrace{\textcolor{orange}{1} \cdot \binomseq{n}{r}}_{5^{L-i}},
        \underbrace{\textcolor{orange}{3} \cdot \binomseq{n}{r}}_{5^{L-i}},
        \underbrace{\textcolor{orange}{1} \cdot \binomseq{n}{r}}_{5^{L-i}}
    \right\}, \\  %\text{ if } \ 5^{i-1}\leq r<5^{i}  \text{ for } i=1,2, \ldots L.\\
    \binomseq{n}{\textcolor{purple}{3} \cdot 5^L+r}& =\left\{\underbrace{\textcolor{purple}{0} \cdot \binomseq{n}{r}}_{5^{L-i}},\underbrace{\textcolor{purple}{0} \cdot\binomseq{n}{r}}_{5^{L-i}}, \underbrace{\textcolor{purple}{0} \cdot\binomseq{n}{r}}_{5^{L-i}},\underbrace{\textcolor{purple}{1} \cdot\binomseq{n}{r}}_{5^{L-i}},\underbrace{\textcolor{purple}{4} \cdot\binomseq{n}{r}}_{5^{L-i}}\right\}, \\
   % \text{ if } \ 5^{i-1}\leq r<5^{i}  \text{ for } i=1,2, \ldots L.  \\
    \binomseq{n}{\textcolor{green}{4} \cdot 5^L+r}& =\left\{\underbrace{\textcolor{green}{0} \cdot\binomseq{n}{r}}_{5^{L-i}},\underbrace{\textcolor{green}{0} \cdot\binomseq{n}{r}}_{5^{L-i}}, \underbrace{\textcolor{green}{0} \cdot\binomseq{n}{r}}_{5^{L-i}}, \underbrace{\textcolor{green}{0} \cdot\binomseq{n}{r}}_{5^{L-i}},  \underbrace{\textcolor{green}{1} \cdot\binomseq{n}{r}}_{5^{L-i}}\right\}.
  % \text{ if } \ 5^{i-1}\leq r<5^{i}  \text{ for } i=1,2, \ldots L.
   \end{align*}
    \end{enumerate}

Note that sequences of the form $\textbinomseq{n}{t \cdot 5^L+r}$, for $t=1,2,3,4$,  given in the previous remark, are constructed from multiples of the sequences $\textbinomseq{n}{r}\}$, where the coefficients are given by the elements of the sequence $\textbinomseq{n}{t}$, which appear in Table~\ref{tab:p5} 
(a table with the first 25 binomial 5-ary sequences can be found in the appendix).
  A similar result can be obtained for the sequences of the form $\textbinomseq{n}{t \cdot 5^L}$, for $t=1,2,3,4$, where these sequences are obtained directly through the terms of the  binomial sequences $\textbinomseq{n}{t}$. 
  %Observe that we can rewrite the expressions~\ref{seqa5.1}-\ref{seqa5.4} as follow:
 % \begin{align*}
 %   \binomseq{n}{\textcolor{blue}{1}\cdot5^L}\right\}& =\color{blue}\left\{\underbrace{00\ldots00}_{ 5^{L}}, \underbrace{11\ldots11}_{5^{L}},\underbrace{22\ldots22}_{5^{L}}, \underbrace{33\ldots33}_{5^{L}}, \underbrace{44\ldots44}_{5^{L}}\right\}  \\
 %  \binomseq{n}{\textcolor{orange}{2} \cdot 5^L}\right\}& =\left\{\color{orange}\underbrace{00\ldots00}_{5^{L}}, \underbrace{00\ldots00}_{5^{L}}, \underbrace{11\ldots11}_{5^{L}},\underbrace{33\ldots33}_{5^{L}}, \underbrace{11\ldots11}_{5^{L}}\color{black}\right\} \\
 %   \binomseq{n}{3 \cdot 5^L}\right\}& =\left\{\underbrace{00\ldots00}_{3 \cdot 5^{L}}, \underbrace{11\ldots11}_{5^{L}},\underbrace{44\ldots44}_{5^{L}}\right\}  \\
 %   \binomseq{n}{4 \cdot 5^L}\right\}& =\left\{\underbrace{00\ldots00}_{4 \cdot 5^{L}}, \underbrace{11\ldots11}_{5^{L}} \right\}
 %  \end{align*}
 This observation will help us to generalise $p$-ary binomial  sequences in the next subsection.
 
% Note that, the $5$-ary sequences are based on the first $4$ $5$-ary sequences $\binom{n}{1}, \binom{n}{2}, \binom{n}{3}, \binom{n}{4}$. 
As a consequence of the previous remark, we have that the sequence $\textbinomseq{n}{5^L-1}$ with $L\geq 1$, is of the form $\{\underbrace{00\ldots00}_{5^L-1} 1\}$ with period $5^L$.
% the following result.
% \begin{corollary}
%     The sequence $\textbinomseq{n}{5^L-1}$ with $L\geq 1$, is of the form $\{\underbrace{00\ldots00}_{5^L-1} 1\}$ with period $5^L$.
% \end{corollary}

%\begin{itemize}
    % \item The period of $\binom{n}{k}$, where $k=5^t+r$, is $T=5^{t+1}$.
    % \item 
    % $\binom{n}{5^t}=\underbrace{00\ldots00}_{5^t} \underbrace{11\ldots11}_{5^t}\underbrace{22\ldots 22}_{5^t}\underbrace{3\ldots 3}_{5^t}
    % \underbrace{44\ldots 44}_{5^t}$
    %     \item 
    % $\binom{n}{5^t-1}=\underbrace{00\ldots00}_{5^t-1} 1$

    %  \item 
    % $\binom{n}{5^t+k}=[\underbrace{00\ldots00}_{5^{t}}, \binom{n}{k} ,2\binom{n}{k},3\binom{n}{k}, 4\binom{n}{k}], k<5^t$

% \item 
%     $\binom{n}{5^t+k}=[\underbrace{00\ldots00}_{2\cdot 5^{t}},  \binom{n}{k}, 3\binom{n}{k} , \binom{n}{k}], 5^t<k<2\cdot 5^t$
    %    \item 
    % $\binom{n}{5^t+k}=[\underbrace{00\ldots00}_{3\cdot 5^{t}}, \binom{n}{k},4\binom{n}{k} ],  3\cdot 5^t<k<4\cdot 5^t$

    %        \item 
%     % $\binom{n}{5^t+k}=[\underbrace{00\ldots00}_{4\cdot 5^{t}}, \binom{n}{k}], 4\cdot   5^t<k< 5^{t+1}$
% \end{itemize}

\subsection{Sequences over $\mathbb{F}_p$} \label{Subsect:F_p}
 In this section, we study the properties of the binomial sequences over $\mathbb{F}_p$ generalising some of the results obtained in the previous subsections. 

\subsubsection{Structure and period}
 We start analysing the formation rules and period of this family of sequences.

\begin{theorem}\label{thm:period}
    The period of the sequence $\textbinomseq{n}{p^L+k}$, where $0\leq k<p^{L+1}-1$ is $T=p^{L+1}$.
\end{theorem}
\begin{proof}
Firstly, we will prove that any interval of lenght $p^{L+1}$ is equivalent modulo $p$. Then, we divide the sequence in intervals, and the first one is composed by the terms $\binom{0}{p^L+k},\binom{1}{p^L+k},\ldots,\binom{p^{L+1}-1}{p^L+k}$, the second one is composed by the terms $\binom{p^{L+1}}{p^L+k},\binom{p^{L+1}+1}{p^L+k},\ldots,\binom{2p^{L+1}-1}{p^L+k}$, and, in general, the $(m+1)$-th interval is composed by the terms $\binom{mp^{L+1}}{p^L+k},\binom{mp^{L+1}+1}{p^L+k},\ldots,\binom{(m+1)p^{L+1}-1}{p^L+k}$.
Thus, it is enough to prove the congruence
\[
\binom{mp^{L+1}+t}{p^L+k}\equiv\binom{(m+1)p^{L+1}+t}{p^L+k} \bmod p,
\]
for all $0 \leq t < p^{L+1}$. By definition:
\begin{align*}    
    \binom{mp^{L+1}+t}{p^L+k} & %= \frac{(mp^{L+1}+t)!}{(p^L+k)!\,(mp^{L+1}+t-p^L-k)!} 
    = \frac{(mp^{L+1}+t)\cdot (mp^{L+1}+t-1)\cdots (mp^{L+1}+t-p^L-k+1)}{(p^L+k)!}, \\
    \binom{(m+1)p^{L+1}+t}{p^L+k} & %= \frac{((m+1)p^{L+1}+t)!}{(p^L+k)!\,((m+1)p^{L+1}+t-p^L-k)!} 
    = \frac{((m+1)p^{L+1}+t) \cdot ((m+1)p^{L+1}+t-1) \cdots ((m+1)p^{L+1}+t-p^L-k+1)}{(p^L+k)!}.
\end{align*}
Since the following congruences hold for all the integers $i$:
\begin{align*}
    m p^{L+1} + t - i & \equiv t-i \bmod p, \\
    (m+1) p^{L+1} + t - i & \equiv t-i \bmod p,
\end{align*}
then, we have that 
\[
    \binom{mp^{L+1}+t}{p^L+k} 
    \equiv
    \frac{t\cdot (t-1)\cdots (t-p^L-k+1)}{(p^L+k)!}
    \equiv
    \binom{(m+1)p^{L+1}+t}{p^L+k} \bmod p.
\]
Hence, we have proven that $\binom{n}{p^L+k} \equiv \binom{n+p^{L+1}}{p^L+k}\bmod p$, for $n \geq 0$. 

Finally, we will prove that there is no $L' < L+1$ such that $\binom{n}{p^L+k} \equiv \binom{n+p^{L'}}{p^L+k} \bmod p$. But since the period divides $p^{L+1}$ and the first  $p^{L}$ bits of the sequence are 0s, it is enough to show that in the next $(p-1)p^L$ terms there is at least one term different from 0. For example, $\binom{p^L+k}{p^L+k} \equiv 1 \bmod p$ for all $k < (p-1)p^L$. Then, the period must be $T=p^{L+1}$.
\end{proof}

It is important to note that the first $p$ binomial sequences $\textbinomseq{n}{t}$, with $t=0, \ldots, p-1$, play a crucial role in the construction of the remaining $p$-ary sequences and in their properties, as we will see in Section~\ref{Sec:decomposition}. 
The next result is a consequence of the previous theorem and establishes the period of the first $p-1$ binomial sequences (excluding $\textbinomseq{n}{0}$). 
Although this result is quite obvious, it plays an important role in understanding subsequent developments.
\begin{corollary}
    The period of the sequence $\textbinomseq{n}{i}$, for $i=1,2, \ldots, p-1$ is $T=p$.
\end{corollary}

Next, we introduce a minor but significant result concerning sequences of the form $\textbinomseq{n}{p^L-1}$. As observed in %Tables~\ref{tab:1} and 
Table \ref{tab:2}, these sequences correspond to the final  sequence before a change in period occurs. They mark the transition point in the family of sequences when organized by increasing order. Moreover, the structure of these sequences remains fixed, regardless of the value of $L$. 
\begin{theorem}
    The sequence $\textbinomseq{n}{p^L-1}$ is of the form $\{00\ldots 01\}$ with period $p^L$.
\end{theorem}

\begin{proof}
We know that the first $p^L-1$ elements of the sequence $\textbinomseq{n}{p^L-1}$ 
are zeros and the element in the position $p^L$ is $\textbinomseq{p^L-1}{p^L-1}=1$, therefore the first $p^L$ elements of the sequence are of the form:
$\{\underbrace{0 \ 0   \ldots   0}_{p^{L}-1\text{ zeros}} 1\}$
According to Theorem \ref{thm:period}, the sequence has period $p^L$.
The theorem follows. 
\end{proof}

Next result provides a way to obtain directly binomial sequences. The proof of this result is omitted since, in Section~\ref{Sec:decomposition}, we introduce an alternative algorithm to construct these sequences and where we will prove their structure.
\begin{theorem} \label{th:bin_p}
Let $k$ be a positive integer and $L, r$ be non-negative integers. For $j \in \{1,2,\ldots, p-1\}$, consider $\textbinomseq{n}{j}=\{s^{(j)}_0,s^{(j)}_1,\ldots, s^{(j)}_{p-1}\}$.
The sequences $\textbinomseq{n}{k}$, where $k=p^L+r$ and $0 \leq r<(p-1)\cdot p^L$ have the following forms:
\bigskip
  \begin{enumerate}[(a),topsep=0pt,partopsep=0pt,parsep=0pt,itemsep=0pt]
      \item For the multiples of the power of $p$, we have that:
  \begin{equation*}
  \binomseq{n}{j \cdot p^L} = \left\{\underbrace{s^{(j)}_0\ldots s^{(j)}_0}_{p^{L}}, \underbrace{s^{(j)}_1\ldots s^{(j)}_1}_{p^{L}}, \ldots, \underbrace{s^{(j)}_{p-1}\ldots s^{(j)}_{p-1}}_{p^{L}}\right\}, \label{seqap.1}
  \end{equation*}
  for any $j \in \{1,2,\ldots, p-1\}$.
  \item For $i=1,2, \ldots L$ and $5^{i-1}\leq r<5^{i}$, we have that
  \begin{equation*}
  \binomseq{n}{\textcolor{blue}{j}\cdot p^L+r} = \left\{\underbrace{\textcolor{blue}{s^{(j)}_0} \cdot \binomseq{n}{r}}_{p^{L-i}}, \underbrace{\textcolor{blue}{s^{(j)}_1} \cdot \binomseq{n}{r}}_{p^{L-i}} ,\ldots, \underbrace{\textcolor{blue}{s^{(j)}_{p-1}} \cdot \binomseq{n}{r}}_{p^{L-i}}\right\}, \label{seqbp.1}
   % \text{ if } \ 5^{i-1}\leq r<5^{i}  \text{ for } i=1,2, \ldots L\\
   \end{equation*}
 for any $j \in \{1,2,\ldots, p-1\}$.
    \end{enumerate}
\end{theorem}

As a consequence of the previous theorem, we have the following result.
\begin{corollary}
    The period of 
     $\textbinomseq{n}{p^L}$  is $T=p^{L+1}$ and its form is 
     \[\binomseq{n}{p^L} = \left\{ 
            \underbrace{0\ldots0}_{p^{L}}, 
            \underbrace{1\ldots1}_{p^{L}}, 
            \underbrace{2\ldots2}_{p^{L}},
            \ldots,
            \underbrace{p-1\ldots p-1}_{p^{L}}
        \right\}. 
     \]
\end{corollary}

\subsubsection{Characteristic polynomial and linear complexity}

Now, we are ready to study the characteristic polynomial of these sequences. First, we need to introduce some preliminary results.
 
\begin{lemma}\cite[Lemma 8]{Cardell2019a}.\label{lem:pascal_rule} Given two positive integers $n$ and $k$, with $n > k \geq 1$, we have the following:
\[  \binom{n+1}{k} = \binom{n}{k} + \binom{n}{k-1}. \]
\end{lemma}

 \begin{remark}
Lemma \ref{lem:pascal_rule} provides an concrete definition of the shifting operator $\textbf{E}$ for binomial sequences: the action of $\textbf{E}$ on the $n$-th term of the $k$-th binomial sequence is given by
\[ E \binom{n}{k} = \binom{n+1}{k} = \binom{n}{k} + \binom{n}{k-1}. \]
Moreover, applying Lemma~\ref{lem:pascal_rule} repeatedly, the action of the operator $\textbf{E}^r = \underbrace{\textbf{E} \!\,\, \circ \cdots \circ \!\,\, \textbf{E}}_{r\text{ copies}}$ on the $n$-th term of the $k$-th binomial sequence is
% \begin{equation}\label{eq:r-shifting-old} 
% \color{red}E^r \binom{n}{k} = \binom{n+r}{k} = \binom{n}{k} + \binom{n}{k-1} + \binom{n+1}{k-1} + \binom{n+2}{k-1} + \cdots + \binom{n+r-1}{k-1}. 
% \end{equation}
 
\begin{align}\label{eq:r-shifting} 
    E^r \binom{n}{k} = \binom{n+r}{k} & =
    \binom{n+r-1}{k}+\binom{n+r-1}{k-1} \nonumber \\
    & = \binom{n+r-2}{k}+\binom{n+r-2}{k-1}+\binom{n+r-1}{k-1} \nonumber \\
    & = \binom{n+r-3}{k}+\binom{n+r-3}{k-1}+\binom{n+r-2}{k-1}+\binom{n+r-1}{k-1} \nonumber \\
    & = \cdots \nonumber \\
    & = \binom{n}{k} + \binom{n}{k-1} + \binom{n+1}{k-1} + \binom{n+2}{k-1} + \cdots + \binom{n+r-1}{k-1}.
\end{align}
\end{remark}

% \textcolor{purple}{Notice that $\textbinomseq{n+1}{k}$ is a shifted version of the sequence $\textbinomseq{n}{k}$.}

Next result allows us to obtain a binomial sequence through the next binomial sequence and a shifted version by one position of this.
\color{black}
\begin{lemma}\label{prop:binom_seq_equivalence} 
For any $k \geq 1$, the binomial sequence $\textbinomseq{n}{k-1}$ is the same  as $\textbinomseq{n+1}{k} - \textbinomseq{n}{k}$.    
\end{lemma}
\begin{proof}
 The proof follows directly from Lemma~\ref{lem:pascal_rule}. 
\end{proof}

 \begin{example}
Consider the binomial sequences over $\mathbb{F}_3$.
   It is possible to see that the sequence
   $\textbinomseq{n}{2}$
   is the same as
   $$
   \binomseq{n+1}{3}-\binomseq{n}{3} = 
   \{001112220\}-\{000111222\}=
   \{001001001\}.
   $$
\end{example}

% \color{red}
% \begin{theorem}\cite[Theorem 1]{CardellFuster2016}\label{th:sum}
%  Let $\{s_i\}_{i\geq 0}$ be a binary sequence whose characteristic polynomial
% is $(x-1)^n$. Then, the characteristic polynomial of the sequence  $\{u_i\}_{i\geq 0}$, where $u_i = s_i + s_{i+1}$ , is $(1+x)^{n-1}$.
% \end{theorem}

% \todomiguel{Este teorema es el que hemos mejorado en Corollary \ref{cor:iff_charpol_binomials} y ya podemos quitarlo}

The following lemma generalizes \cite[Lemma 2]{CardellFuster2016}

\begin{lemma}\label{lem:iff_generator_polynomials} Let $\{ s_n \} $  and $\{ u_n \} $ be  two $p$-ary sequences such that  $u_n = s_{n+1} - s_n$. Then, $q(x) \in \mathbb{F}_p[x]$ generates $\{ u_n \} $ if and only if $(x-1) q(x)$ generates $\{ s_n \} $.    
\end{lemma}

% \color{red}

% \begin{proof} Assume $q(x)$ is of the form $q(x) = \sum_{i=0}^{k-1} a_i x^i$ for certain $k$. Then $(x-1) q(x) = \sum_{i=0}^{k-1} a_i(x^{i+1}-x^i)$.

% If $q(x)$ generates $\{ u_n \} $ then $\sum_{i=0}^{k-1} a_i u_{i+j} = 0$ for all $j \geq 0$. Hence, by the definition of $\{ u_n \} $, we have $\sum_{i=0}^{k-1} a_i (s_{i+1+j} - s_{i+j}) = 0$, so $(x-1) q(x)$ generates $\{ s_n \} $.

% Conversely, if $(x-1)q(x)$ generates $\{ s_n \} $ then $\sum_{i=0}^{k-1} a_i (s_{i+1+j} - s_{i+j}) = 0$ for all $j \geq 0$. Hence, by the definition of $\{ u_n \} $, we have $\sum_{i=0}^{k-1} a_i u_i = 0$, so $q(x)$ generates $\{ u_n \} $.
% \end{proof}

% \color{blue}

\begin{proof} Notice that $u_n = (E-\id) s_n$. 

If $q(x)$ generates $\{u_n\}$ then $q(E)u_n = 0$ for all $n \geq 0$. Therefore, $(E-\id)q(E)s_i = q(E) (E-\id)s_i = q(E)u_i = 0$ for all $n \geq 0$ so $(x-1) q(x)$ generates $\{ s_n \} $.

Conversely, if $(x-1)q(x)$ generates $\{s_n\} $ then $(E-\id) q(E) s_n = 0$ for all $n \geq 0$. Therefore, $q(E)u_n = q(E)(E-\id)s_n = 0$ for all $n \geq 0$ and $q(x)$ generates $\{ u_n \} $.
\end{proof}

%Notice that the polynomial $(x-1)^k$ can be expressed as
%\[ (x-1) (x-1)^{k-1} = (x-1) \sum_{i=0}^{k-1} \binom{k-1}{i} (-1)^{k-1-i} x^i = \sum_{i=0}^{k-1} \binom{k-1}{i} (-1)^{k-1-i} (x^{i+1} - x^i). \]

The following theorem generalizes \cite[Theorem 1]{CardellFuster2016}

\begin{theorem}\label{thm:characteristic_u_characteristic_s} Let $\{ s_n \} $ be a $p$-ary sequence and let $\{ 
u_n \} $ be the $p$-ary sequence defined as $u_n = s_{n+1} - s_n$, for all $n \geq 0$. If $(x-1) m(x)$ is the characteristic polynomial of $\{s_n\}$ then $m(x)$ is the characteristic polynomial of $\{ u_n \} $.    
\end{theorem}

\begin{proof} Suppose that $(x-1) m(x)$ is the characteristic polynomial of $\{ s_n \} $. If $m'(x)$ is a polynomial of degree $\deg(m'(x)) < \deg(m(x))$ that generates $\{ u_n \} $ then Lemma \ref{lem:iff_generator_polynomials} implies that $(x-1) m'(x)$ generates $\{ s_n \} $, with $\deg((x-1)m'(x)) < \deg((x-1)m(x))$, contradicting that $(x-1)m(x)$ is the characteristic polynomial.
\end{proof}

% \todomiguel{Pregunta abierta: el rec\'{i}proco del teorema anterior es cierto? Tengo una prueba a medias, pero falla en un punto, y la intuici\'{o}n me dice que no es cierto.}
% \todovero{Pero necesitamos el recíproco para algo?}
% \todomiguel{No, no lo necesitamos}

The following lemma follows from the fact that the set of all polynomials in $\mathbb{F}_p[x]$ that generate a $p$-ary sequence forms an ideal of $\mathbb{F}_p[x]$.

\begin{lemma}\cite[Corollary 4.1]{GolombGong2005}\label{lem:characteristic_divides_generator} If $q(x)$ is a generator polynomial of a $p$-ary sequence $\{ s_n \} $ and $m(x)$ its characteristic polynomial, then $m(x)$ divides $q(x)$. 
\end{lemma}

Now we can prove the converse of \cite[Theorem 1]{CardellFuster2016}.

\begin{corollary}\label{cor:iff_charpol_binomials} Let $\{ s_n \} $ be a $p$-ary sequence and $\{ 
u_n \} $ be the $p$-ary sequence defined as $u_n = s_{n+1} - s_n$, for all $n \geq 0$. Let $k \geq 1$. Then $(x-1)^{k+1}$ is the characteristic polynomial of $\{ s_n \} $ if and only if $(x-1)^k$ is the characteristic polynomial of $\{ u_n \} $.    
\end{corollary}

\begin{proof} The direct implication follows from Theorem \ref{thm:characteristic_u_characteristic_s}. Assume now that $(x-1)^k$ is the characteristic polynomial of $\{ u_n \} $. Lemma \ref{lem:iff_generator_polynomials} implies that $(x-1)^{k+1}$ generates $\{ s_n \} $. Now, since $\mathbb{F}[x]$ is a unique factorization domain, Lemma \ref{lem:characteristic_divides_generator} implies that the characteristic polynomial is of the form $(x-1)^t$ for certain $1 \leq t \leq k+1$. If $t < k+1$ then Lemma \ref{lem:iff_generator_polynomials} implies that $(x-1)^{t-1}$ generates $\{ u_n \} $ with $t-1 < k$, which contradicts the fact that $(x-1)^k$ is the characteristic polynomial of $\{ u_n \} $. Necessarily $t = k+1$.    
\end{proof}

Now, we are ready to study the linear complexity of the binomial sequences.

\begin{theorem}\label{thm:LC} The characteristic polynomial of the sequence $\textbinomseq{n}{k}$ is $(x-1)^{k+1}$. Hence, its linear complexity $LC$ is $k+1$.
\end{theorem}

\begin{proof}
We prove this result by induction.

For $k=0$,  the characteristic polynomial of the   sequence $\textbinomseq{n}{0} = \{1 1 1 1 1  1 \ldots \,\}$ is $(x-1)$ so $LC=1$.

% For $k=1$, the sequence $\textbinomseq{n}{1} = \{0 \ 1\ 2\ \ldots\ (p-1)\ 0\ 1\ 2\ \ldots\ (p-1)\ \ldots \,\}$ has the characteristic polynomial is $(x-1)^2$ so $LC_1=2$. 
For $k=1$, as we can express $\textbinomseq{n}{0} = \textbinomseq{n+1}{1}-\textbinomseq{n}{1}$ then, for  Corollary~\ref{cor:iff_charpol_binomials}, we have that the characteristic polynomial of $\textbinomseq{n}{1}$ is $(x-1)^2$ and $LC=2$.

Suppose that the characteristic polynomial of the sequence $\textbinomseq{n}{k-1}$ is $(x-1)^k$, so it has $LC=k$. According to Lemma~\ref{prop:binom_seq_equivalence}, we have that
%the sequence represented by $\binom{n}{k}$ can be obtained from sequence $\binom{n}{k+1}$.
$\textbinomseq{n}{k-1} = \textbinomseq{n+1}{k} - \textbinomseq{n}{k}$.
Take $u_n = \binom{n}{k-1}$ and $s_n = \binom{n}{k}$. Now, according to Corollary \ref{cor:iff_charpol_binomials}, the characteristic polynomial of $\textbinomseq{n}{k}$ is $(x-1)^{k+1}$ and hence $LC=k+1$.
\end{proof}

As a consequence of the previous theorem, we have the following result.
\begin{corollary}\label{cor:repr}
Given a sequence with binomial representation $\sum_{i=1}^{t} c_i \textbinomseq{n}{k_i}$, where $k_1<k_2<\cdots< k_t$ are integer indexes and $c_i \neq 0$ for all $i = 1,2,\ldots,t$, then the characteristic polynomial of such a sequence is $(x-1)^{k_t+1}$ and its linear complexity is $k_t+1$.
\end{corollary}

\begin{proof} Theorem \ref{thm:LC} implies that $(x-1)^{k_t+1}$ is a generator polynomial of $\textbinomseq{n}{k_i}$ for all $i = 1,2,\ldots,t$. Thus, we have
\begin{align*}
(E-\id)^{k_t+1} \sum_{i=1}^{t} c_i \binom{n}{k_i} & = \sum_{i=1}^t c_i (E-\id)^{k_t-k_i}(E-\id)^{k_i+1} \binom{n}{k_i} \\
& = \sum_{i=1}^t c_i (E-\id)^{k_t-k_i} 0 = 0, \quad n \geq 0,
\end{align*}
so it generates $\sum_{i=1}^{t} c_i \textbinomseq{n}{k_i}$. Now, let us see that $(x-1)^{k_t}$ do not generate that sequence:
\begin{align*}
(E-\id)^{k_t} \sum_{i=1}^{t} c_i \binom{n}{k_i} & = \sum_{i=1}^{t-1} c_i (E-\id)^{k_t-k_i-1}(E-\id)^{k_i+1} \binom{n}{k_i} + c_t (E-\id)^{k_t} \binom{n}{k_t} \\
& = c_t (E-\id)^{k_t} \binom{n}{k_t} \neq 0, \quad n \geq 0,
\end{align*}
since $(x-1)^{k_t}$ do not generate $\textbinomseq{n}{k_t}$. Hence $(x-1)^{k_t}$ is not a generator polynomial. Necessarily, $(x-1)^{k_t+1}$ is the characteristic polynomial of $\sum_{i=1}^{t} c_i \textbinomseq{n}{k_i}$.
\end{proof}

\color{black}

% \subsubsection{Every sequence of period $p^L$ can be written as a linear combination of binomial sequences of period $p^L$ or less}
% \todosara{pensar este título}

\subsubsection{The linear space of $p$-ary sequences with period $p^L$}

In this section we show that every sequence in $\mathbb{F}_p$ of period $p^L$, for some integer $L \geq 0$, can be written as a linear combination of binomial sequences of period $p^\ell$, with $\ell \leq L$. Moreover, we prove that the set of sequences of period $p^\ell$, with $\ell \leq L$, is a linear space and the set of binomial sequences of period $p^\ell$, with $\ell \leq L$, is a basis of this space.

\begin{theorem}\label{thm:every_sequence_linear_combination} Every sequence of period $p^L$ can be written as a linear combination of binomial sequences of period $p^{\ell}$, with $\ell\leq L$.    
\end{theorem}

\begin{proof}
Let $\{ s_n \} $ be an arbitrary sequence of period $p^L$.
We want to prove that this sequence can be obtained as a linear combination of the first $p^L$ binomial sequences.
%We want to show that there are coefficients $c_0,c_1,c_2,\ldots,c_{p^L-1} \in \mathbb{F}_p$ such that 
%\[
%\{ s_n \} 
%=
%c_0 \binomseq{n}{0}
% +
% c_1 \binomseq{n}{1}
% +
% c_2 \binomseq{n}{2}
% + 
% \cdots 
% +
% c_{p^L-1} \binomseq{n}{p^L-1}
% \]
Since $\{s_n\} $ is periodic of period $p^L$, we just need to find  the coefficients $c_0,c_1,c_2,\ldots,c_{p^L-1} \in \mathbb{F}_p$ such that 
\begin{equation}\label{eq:repr}
\{ s_n \}_{n=0}^{p^L-1}
=
c_0 \binomseq{n}{0}_{n=0}^{p^L-1}
+
c_1 \binomseq{n}{1}_{n=0}^{p^L-1}
+
c_2 \binomseq{n}{2}_{n=0}^{p^L-1}
+ 
\cdots 
+
c_{p^L-1} \binomseq{n}{p^L-1}_{n=0}^{p^L-1}.
\end{equation}
In matrix form, this is equivalent to 
\begin{align*}
\left[ 
\begin{array}{ccccc}
    s_0 & s_1 & s_2 & \ldots & s_{p^L-1}
\end{array}
\right]
&=
\left[
\begin{array}{ccccc}
    c_0 & c_1 & c_2 & \ldots & c_{p^L-1}
\end{array}
\right] H_L
\\
&=
\left[
\begin{array}{ccccc}
    c_0 & c_1 & c_2 & \ldots & c_{p^L-1}
\end{array}
\right]
\left[ 
    \begin{array}{ccccc}
        \binom{0}{0} & \binom{1}{0} & \binom{2}{0} & \cdots & \binom{p^L-1}{0} \\[0.5em]
        \binom{0}{1} & \binom{1}{1} & \binom{2}{1} & \cdots & \binom{p^L-1}{1} \\[0.5em]
        \binom{0}{2} & \binom{1}{2} & \binom{2}{2} & \cdots & \binom{p^L-1}{2} \\
        \vdots & \vdots & \vdots & \ddots & \vdots \\
        \binom{0}{p^L-1} & \binom{1}{p^L-1} & \binom{2}{p^L-1} & \cdots & \binom{p^L-1}{p^L-1}
    \end{array}
\right],\end{align*}
According to
Theorem~\ref{thm:inversaHL},
    $H_L$ is non-singular. Therefore,  this linear system has a unique solution in $\mathbb{F}_p$ given by:
\begin{align}\label{eq:mat}
\left[ 
\begin{array}{ccccc}
    c_0 & c_1 & c_2 & \ldots & c_{p^L-1}
\end{array}
\right]
&=
\left[
\begin{array}{ccccc}
    s_0 & s_1 & s_2 & \ldots & s_{p^L-1}
\end{array}
\right]H_L^{-1}
\nonumber \\
&=\left[
\begin{array}{ccccc}
    s_0 & s_1 & s_2 & \ldots & s_{p^L-1}
\end{array}
\right]
\left[ 
    \begin{array}{ccccc}
        \binom{p^L-1}{p^L-1} & \binom{p^L-1}{p^L-2} & \cdots & \binom{p^L-1}{1} & \binom{p^L-1}{0} \\[0.5em]
        \binom{p^L-2}{p^L-1} & \binom{p^L-2}{p^L-2} & \cdots & \binom{p^L-2}{1} & \binom{p^L-2}{0} \\
        \vdots & \vdots & \ddots & \vdots & \vdots \\[0.25em]
        \binom{1}{p^L-1} & \binom{1}{p^L-2} & \cdots & \binom{1}{1} & \binom{1}{0} \\[0.5em]
        \binom{0}{p^L-1} & \binom{0}{p^L-2} & \cdots & \binom{0}{1} & \binom{0}{0}
    \end{array}
\right].
\end{align}
As a consequence, there exist coefficients $c_0, \ldots, c_{p^L-1}$, such that the expression in \eqref{eq:repr} is valid.
 Notice that the columns in the matrix \eqref{eq:mat} are the binomial sequences mod $p$.
\end{proof}

Now, recall that the set $V(\mathbb{F}_p)$ of all $p$-ary sequences is a linear space. The above theorem shows us the structure of the set of all periodic sequences of period  $p^\ell$ with $\ell\leq L$.

\begin{corollary}\label{cor:vector_space_sequences} Let $W$ be the set of  $p$-ary sequences of period  $p^\ell$ with $\ell\leq L$. Then $W$ is a linear subspace of $V(\mathbb{F}_p)$ of dimension $p^L$ and $\mathcal{B} = \big\{ \textbinomseq{n}{k} \mid \ 0\leq k < p^L \big\}$ is a basis of $W$.
\end{corollary}

\begin{proof} 
It is easy to see that the set $W$ of sequences of period  $p^\ell$ with $\ell\leq L$ is closed under the sum of sequences and the product by an element of $\mathbb{F}_p$, so it is a linear space.
Also, Theorem \ref{thm:every_sequence_linear_combination} shows that every sequence of period $p^\ell$ with $\ell \leq L$ can be written as a linear combination of the sequences $\mathcal{B} = \big\{ \textbinomseq{n}{k} \mid 0\leq k < p^L \big\}$, then $\mathcal{B}$ generates $W$. 
Moreover, since matrix $H_L$ is non-singular (Theorem~\ref{thm:inversaHL}), $\mathcal{B}$ is a basis and hence $W$ is of dimension $p^L$.
\end{proof}

Moreover, as a consequence of Theorem \ref{thm:every_sequence_linear_combination} and Corollary \ref{cor:repr} we deduce the following result.

\begin{theorem} Every $p$-ary sequence of period $p^L$ has a characteristic polynomial of the form $(x-1)^\ell$ and linear complexity $LC = \ell$, with $p^{L-1} < \ell \leq p^L$.
\end{theorem}

\begin{definition}
Representation $\sum_{i=0}^{LC-1}c_i \textbinomseq{n}{i}$ is known as \textbf{binomial representation} of the sequence. 
\end{definition}

As a consequence of the proof of Theorem~\ref{thm:every_sequence_linear_combination}, we can recover the binomial representation from the sequence—and vice versa—using the Pascal-Hadamard matrix (expression~\eqref{eq:mat}). The following example illustrates this relationship.
 
\begin{example}\label{ex:repr}
    Consider for example, the sequence 
    $\{    0   1   2   1   2   2   2   0   2\ldots \}$ with period $T=9$ in $\mathbb{F}_3$.
We can use the matrix in \eqref{eq:mat} to compute the coefficients of the binomial representation of the sequence:
\begin{align*}
\left[
    \begin{array}{ccccccccc}
        0 & 1 & 2 & 1 & 2 & 2 & 2 & 0 & 2
    \end{array}
\right]
&
\left[
\begin{array}{ccccccccc}
\binom{8}{8} & \binom{8}{7} & \binom{8}{6} & \binom{8}{5} & \binom{8}{4} & \binom{8}{3} & \binom{8}{2} & \binom{8}{1} & \binom{8}{0} \\[0.3em]
\binom{7}{8} & \binom{7}{7} & \binom{7}{6} & \binom{7}{5} & \binom{7}{4} & \binom{7}{3} & \binom{7}{2} & \binom{7}{1} & \binom{7}{0} \\[0.3em]
\binom{6}{8} & \binom{6}{7} & \binom{6}{6} & \binom{6}{5} & \binom{6}{4} & \binom{6}{3} & \binom{6}{2} & \binom{6}{1} & \binom{6}{0} \\[0.3em]
\binom{5}{8} & \binom{5}{7} & \binom{5}{6} & \binom{5}{5} & \binom{5}{4} & \binom{5}{3} & \binom{5}{2} & \binom{5}{1} & \binom{5}{0} \\[0.3em]
\binom{4}{8} & \binom{4}{7} & \binom{4}{6} & \binom{4}{5} & \binom{4}{4} & \binom{4}{3} & \binom{4}{2} & \binom{4}{1} & \binom{4}{0} \\[0.3em]
\binom{3}{8} & \binom{3}{7} & \binom{3}{6} & \binom{3}{5} & \binom{3}{4} & \binom{3}{3} & \binom{3}{2} & \binom{3}{1} & \binom{3}{0} \\[0.3em]
\binom{2}{8} & \binom{2}{7} & \binom{2}{6} & \binom{2}{5} & \binom{2}{4} & \binom{2}{3} & \binom{2}{2} & \binom{2}{1} & \binom{2}{0} \\[0.3em]
\binom{1}{8} & \binom{1}{7} & \binom{1}{6} & \binom{1}{5} & \binom{1}{4} & \binom{1}{3} & \binom{1}{2} & \binom{1}{1} & \binom{1}{0} \\[0.3em]
\binom{0}{8} & \binom{0}{7} & \binom{0}{6} & \binom{0}{5} & \binom{0}{4} & \binom{0}{3} & \binom{0}{2} & \binom{0}{1} & \binom{0}{0}
\end{array}
\right]\\
=
\left[
\begin{array}{ccccccccc}
    0 & 1 & 2 & 1 & 2 & 2 & 2 & 0 & 2
    \end{array}
\right]
&
\left[
\begin{array}{ccccccccc}
1 & 2 & 1 & 2 & 1 & 2 & 1 & 2 & 1 \\
0 & 1 & 1 & 0 & 2 & 2 & 0 & 1 & 1 \\
0 & 0 & 1 & 0 & 0 & 2 & 0 & 0 & 1 \\
0 & 0 & 0 & 1 & 2 & 1 & 1 & 2 & 1 \\
0 & 0 & 0 & 0 & 1 & 1 & 0 & 1 & 1 \\
0 & 0 & 0 & 0 & 0 & 1 & 0 & 0 & 1 \\
0 & 0 & 0 & 0 & 0 & 0 & 1 & 2 & 1 \\
0 & 0 & 0 & 0 & 0 & 0 & 0 & 1 & 1 \\
0 & 0 & 0 & 0 & 0 & 0 & 0 & 0 & 1 \\
\end{array}
\right]
=
\left[
    \begin{array}{ccccccccc}
        0 & 1 & 0 & 1 & 0 & 2 & 0 & 0 & 0
    \end{array}
\right].
\end{align*}
Therefore, this sequence can be computed as:
\[
\binomseq{n}{1} + \binomseq{n}{3} + 2 \cdot
\binomseq{n}{5} 
=
\{0 1 2 0 1 2 0 1 2 \}+\{0 0 0 1 1 1 2 2 2 \}+2\cdot \{0 0 0 0 0 1 0 0 2 \}=\{   0   1   2   1   2   2   2   0   2\},
\]
and hence, its binomial representation is $\textbinomseq{n}{1} + \textbinomseq{n}{3} + 2 \cdot \textbinomseq{n}{5}$.
\end{example}

A natural question arises when we think about the binomial representation of shifted versions of one sequence. The next theorem gives us a method to obtain the binomial representation of such shifts. Its proof is a direct consequence of Lemma~\ref{lem:pascal_rule}.

\color{black}
\begin{theorem}\label{thm:bin}
Let $\sum_{i=0}^{LC-1} c_i \textbinomseq{n}{i}$ with $c_i \in \Fset_p$ be the binomial representation of a sequence. If we shift cyclically such a sequence one bit to the left, then its binomial representation is:
\begin{equation*}\label{eq:repr2}
\sum_{i=1}^{LC-1}c_i \left[\left\{\!\binom{n}{i}+\binom{n}{i-1}\!\right\} \right]+c_0 \binomseq{n}{0}.
\end{equation*}
\end{theorem}
 
\begin{example}
    Consider again the sequence in Example~\ref{ex:repr}.
    If we shift this sequences one position to the left, we obtain the following sequence:
    $\{121222020 \ldots\}$.
    According to Theorem~\ref{thm:bin}, we can compute the binomial representation as:
    $$
    \left( \binomseq{n}{1} + \binomseq{n}{0} \right) + \left( \binomseq{n}{3} + \binomseq{n}{2} \right) + 2 \left( \binomseq{n}{5} + \binomseq{n}{4} \right) =
\{1   2   1   2   2   2   0   2   0 \}.
    $$
\end{example}
\color{black}

The next result gives us the matrix of the shift operator for the sequences of period $p^\ell$ with $\ell\leq L$. 
The proof is a direct consequence of Theorem \ref{thm:bin}.

% \begin{corollary} Let $W$ be the linear space of $p$-ary sequences of period  $p^\ell$ with $\ell\leq L$ and let $E|_W: W \rightarrow W$ be the restriction of the shift operator to $W$. The matrix associated to $E|_W$ with respect to the basis $\mathcal{B} = \big\{ \binomseq{n}{k}\ | \  k=0, 1, \ldots, p^L-1 \big\}$ is \todosara{Si dominio e contradominio es $W$ creo que no es necesario colocar la restricción}
% \begin{equation}\label{eq:E_matrix}
% M_\mathcal{B}(E|_W) = \left[ 
%     \begin{array}{cccccc}
%         1 & 1 \\
%           & 1 & 1 \\
%           &  & \ddots & \ddots \\
%           &&& 1 & 1 \\
%           &&&&1
%     \end{array}
% \right] \in \mathbb{F}_p^{p^L \times p^L}.
% \end{equation}
% \end{corollary}

\begin{corollary} \label{cor:base}
Let $W$ be the linear space of $p$-ary sequences of period  $p^\ell$ with $\ell\leq L$ and let $\textbf{E}|_{W}: W \rightarrow W$ be the shift operator in $W$. The matrix associated to $\textbf{E}|_{W}$ with respect to the basis $\mathcal{B} = \big\{ \textbinomseq{n}{k} \mid k=0, 1, \ldots, p^L-1 \big\}$ is 
\begin{equation*}%\label{eq:E_matrix}
M_\mathcal{B}(\textbf{E}|_{W}) = \left[ 
    \begin{array}{cccccc}
        1 & 1 \\
          & 1 & 1 \\
          &  & \ddots & \ddots \\
          &&& 1 & 1 \\
          &&&&1
    \end{array}
\right] \in \mathbb{F}_p^{p^L \times p^L}.
\end{equation*}
\end{corollary}

It turns out that the matrix $M_\mathcal{B}(\textbf{E}|_{W})$ exhibits a strong connection with binomial coefficients, a relationship that becomes evident upon computing its powers. The following corollary follows from straightforward computations.

\begin{corollary}\label{cor:Er} For any integer $r \in \{ 0,1,\ldots,p^L-1 \}$ we have
\begin{equation*}%\label{eq:E_matrix}
M_\mathcal{B}(\textbf{E}^r|_{W}) = M_\mathcal{B}(\textbf{E}|_{W})^r =
\left[ 
    \begin{array}{ccccccccccccccccc}
        \binom{r}{0} & \binom{r}{1} & \binom{r}{2} & \cdots & \cdots & \binom{r}{r-1} & \binom{r}{r} \\[0.3em]
        & \binom{r}{0} & \binom{r}{1} & \binom{r}{2} & \cdots & \cdots & \binom{r}{r-1} & \binom{r}{r} \\[0.3em]
        && \ddots & \ddots & \ddots &&& \ddots & \ddots \\[0.3em]
        &&& \binom{r}{0} & \binom{r}{1} & \binom{r}{2} & \cdots & \cdots & \binom{r}{r-1} & \binom{r}{r} \\[0.3em]
        &&&& \binom{r}{0} & \binom{r}{1} & \binom{r}{2} & \cdots & \cdots & \binom{r}{r-1} \\[0.3em]
        &&&&& \binom{r}{0} & \binom{r}{1} & \cdots & \cdots & \binom{r}{r-2} \\[0.3em]
        &&&&&& \ddots & \ddots & \ddots & \vdots \\[0.3em]
        &&&&&&& \binom{r}{0} & \binom{r}{1} & \binom{r}{2} \\[0.3em]
        &&&&&&&& \binom{r}{0} & \binom{r}{1} \\[0.3em]
        &&&&&&&&& \binom{r}{0} \\
    \end{array}
\right] \in \mathbb{F}_p^{p^L \times p^L},
\end{equation*}
and $M_\mathcal{B} \left( \textbf{E}^{p^L}|_{W} \right) = I_{p^L} \in \mathbb{F}_p^{p^L \times p^L}$.
\end{corollary}

The matrix presented in Corollary~\ref{cor:Er} is already known in the literature and has appeared in other contexts, such as Coding Theory (see, for instance, \cite{SmarandacheGluessingLuerssenRosenthal2002}).

Moreover, from \eqref{eq:r-shifting} and Corollary \ref{cor:Er}, we can deduce the following result.

\begin{theorem}
Given a prime integer $p$ and an integer $k \geq 1$, we have that for all $n$ it holds: 
\[
\binom{n+r}{k} 
\equiv \binom{n}{k}+\sum_{j=0}^{r-1}\binom{n+i}{k-1}
\equiv
\sum_{j=0}^{\min\{r,k\}} \binom{r}{j} \binom{n}{k-j}
\bmod p.
\] 
\end{theorem}

\begin{proof} The shift operator acts as
\begin{equation}\label{eq:r-shifting-seq_1} 
    \textbf{E}^r|_W \left( \binomseq{n}{k} \right) = \binomseq{n+r}{k}. 
\end{equation}
Equation \eqref{eq:r-shifting} implies that
\begin{equation}\label{eq:r-shifting-seq_2} 
    \textbf{E}^r|_W \left( \binomseq{n}{k} \right) = \left\{ \binom{n}{k} + \binom{n}{k-1} + \binom{n+1}{k-1} + \binom{n+2}{k-1} + \cdots + \binom{n+r-1}{k-1} \right \}. 
\end{equation}
Now, let us denote by $C_\mathcal{B}(\textbf{u})$ the coordinates of the vector $\textbf{u}$ in the basis $\mathcal{B}$ (see Corollary~\ref{cor:base}). Then
\begin{align*}
C_\mathcal{B} & \left(\textbf{E}^r|_W \left( \binomseq{n}{k} \right) \right)
=
M_\mathcal{B}(\textbf{E}^r|_{W}) C_\mathcal{B} \left( \binomseq{n}{k} \right) \\
& = 
 \left[ 
    \begin{array}{ccccccccccccccccc}
        \binom{r}{0} & \binom{r}{1} & \binom{r}{2} & \cdots & \cdots & \binom{r}{r-1} & \binom{r}{r} \\[0.3em]
        & \binom{r}{0} & \binom{r}{1} & \binom{r}{2} & \cdots & \cdots & \binom{r}{r-1} & \binom{r}{r} \\[0.3em]
        && \ddots & \ddots & \ddots &&& \ddots & \ddots \\[0.3em]
        &&& \binom{r}{0} & \binom{r}{1} & \binom{r}{2} & \cdots & \cdots & \binom{r}{r-1} & \binom{r}{r} \\[0.3em]
        &&&& \binom{r}{0} & \binom{r}{1} & \binom{r}{2} & \cdots & \cdots & \binom{r}{r-1} \\[0.3em]
        &&&&& \binom{r}{0} & \binom{r}{1} & \cdots & \cdots & \binom{r}{r-2} \\[0.3em]
        &&&&&& \ddots & \ddots & \ddots & \vdots \\[0.3em]
        &&&&&&& \binom{r}{0} & \binom{r}{1} & \binom{r}{2} \\[0.3em]
        &&&&&&&& \binom{r}{0} & \binom{r}{1} \\[0.3em]
        &&&&&&&&& \binom{r}{0} \\
    \end{array}
\right]
\left[ 
    \begin{array}{ccccccccccccccccc}
        0 \\[0.3em]
        0 \\[0.3em]
        \vdots \\[0.3em]
        0 \\[0.3em]
        1 \\[0.3em]
        0 \\[0.3em]
        \vdots \\[0.3em]
        0 \\[0.3em]
        0 \\[0.3em]
        0 \\
    \end{array}
\right]
\left.
    \begin{array}{ccccccccccccccccc}
        \phantom{0} \\[0.3em]
        \phantom{0} \\[0.3em]
        \phantom{\vdots} \\[0.3em]
        \phantom{0} \\[0.3em]
        \text{$\leftarrow (k+1)$-th position}. \\[0.3em]
        \phantom{0} \\[0.3em]
        \phantom{\vdots} \\[0.3em]
        \phantom{0} \\[0.3em]
        \phantom{0} \\[0.3em]
        \phantom{0} \\
    \end{array}
\right.
\end{align*}
If $k \geq r$, we have that 
\[
C_\mathcal{B} \left(\textbf{E}^r|_W \left( \textbinomseq{n}{k} \right) \right) 
= 
\left[ 
    \begin{array}{c}
        0 \\ 
        \vdots \\
        0 \\[0.3em]
        \binom{r}{r} \\[0.5em]
        \binom{r}{r-1} \\
        \vdots \\[0.3em]
        \binom{r}{0} \\[0.5em]
        0 \\ 
        \vdots \\
        0 \\
    \end{array}
\right]
\left.
    \begin{array}{l}
        \phantom{0} \\ 
        \phantom{\vdots} \\
        \phantom{0} \\[0.3em]
        \text{$\leftarrow (k-r+1)$-th position} \\
         \phantom{0} \\[0.3em]
        \phantom{\vdots} \\[0.3em]
        \text{$\leftarrow (k+1)$-th position} \\[0.3em]
        \phantom{0} \\ 
        \phantom{\vdots} \\
        \phantom{0} \\
    \end{array}
\right.
\]
so
\begin{align}\label{eq:r-shifting-seq_3a} 
\textbf{E}^r|_W \left( \binomseq{n}{k} \right) 
 =  &
\binom{r}{r} \binomseq{n}{k-r} 
    + \binom{r}{r-1} \binomseq{n}{k-r+1}
    + \cdots  \\
    & 
    \phantom{.aaaaaa}+ \binom{r}{1} \binomseq{n}{k-1}
    + \binom{r}{0} \binomseq{n}{k} \nonumber \\
 =  &
\left\{ 
    \binom{r}{r} \binom{n}{k-r}
    + \binom{r}{r-1} \binom{n}{k-r+1}
    + \cdots \right. \nonumber\\
    & 
   \left. \phantom{aaaaaa} + \binom{r}{1} \binom{n}{k-1}
    + \binom{r}{0} \binom{n}{k}
\right\}.
\end{align}
If $k < r$, then
\[
C_\mathcal{B} \left(\textbf{E}^r|_W \left( \textbinomseq{n}{k} \right) \right) 
= 
\left[ 
    \begin{array}{c}
        \binom{r}{k} \\[0.3em]
        \binom{r}{k-1} \\
        \vdots \\[0.3em]
        \binom{r}{0} \\[0.3em]
        0 \\ 
        \vdots \\
        0 \\
    \end{array}
\right]
\left.
    \begin{array}{l}
        \phantom{0} \\[0.3em]
         \phantom{0} \\[0.3em]
        \phantom{\vdots} \\[0.3em]
        \text{$\leftarrow (k+1)$-th position} \\[0.3em]
        \phantom{0} \\
        \phantom{0} \\
        \phantom{0} \\
    \end{array}
\right.
\]
and in this case
\begin{align}\label{eq:r-shifting-seq_3b} 
\textbf{E}^r|_W \left( \binomseq{n}{k} \right) 
& = 
\binom{r}{k} \binomseq{n}{0}
    + \binom{r}{k-1} \binomseq{n}{1}
    + \cdots
    + \binom{r}{1} \binomseq{n}{k-1}
    + \binom{r}{0} \binomseq{n}{k} \nonumber \\
& = 
\left\{ 
    \binom{r}{k} \binom{n}{0}
    + \binom{r}{k-1} \binom{n}{1}
    + \cdots
    + \binom{r}{1} \binom{n}{k-1}
    + \binom{r}{0} \binom{n}{k}
\right\}.
\end{align}
Comparing the terms of the sequences given in \eqref{eq:r-shifting-seq_1}, \eqref{eq:r-shifting-seq_2}, \eqref{eq:r-shifting-seq_3a} and \eqref{eq:r-shifting-seq_3b} yields the result.
\end{proof}

As a corollary of the previous theorem, we deduce how the shifting operator acts on a sequence of period $p^L$.

\begin{corollary} Let $\{ s_n \}$ be a periodic $p$-ary sequence of period $p^L$ with binomial representation $\sum_{i=0}^{LC-1}c_i \textbinomseq{n}{i}$. Then
\begin{align*}
\{ s_{n+r} \} 
\equiv
\textbf{E}^r|_W(\{ s_n \})
& \equiv
\sum_{i=0}^{LC-1} c_i \binomseq{n+r}{i} \bmod p \\
& \equiv
\sum_{i=0}^{LC-1} c_i \left\{
\binom{n}{i}+\sum_{j=0}^{r-1}\binom{n+j}{i-1}
\right\} \bmod p \\
& \equiv
\sum_{i=0}^{LC-1} c_i \left\{
\sum_{j=0}^{\min\{r,k\}} \binom{r}{j} \binom{n}{i-j}
\right\}
\bmod p
.
\end{align*}
    
\end{corollary}

\color{black}
 
% \todomiguel{Revisad si os gusta lo que he escrito (en azul). He añadido la última congruencia (que es creo que es válida incluso sin tomar módulo p) en marrón para rematar la sección.}
% \todosara{si sumas secuencias, tienes que poner el módulo, si no como sabes que el resultado va a ser menor que p?}
% \todomiguel{Voy a escribirlo en módulo p}
% \todosara{no veo nada en marron aqui}
% \todosara{Lo marrón era lo que estaba dos páginas arriba}

\section{A $p$-ary expansion decomposition of the binomial sequence} \label{Sec:decomposition}

In this section we show that all the binomial sequences modulo $p$ can be obtained via a simple and efficient algorithm and basic operations modulo $p$. Rather than computing large binomial coefficients and taking modulo $p$, only two simple operations are required: the Hadamard product of two sequences and the expansion of a sequence. Before, we recall Lucas's theorem.

\begin{theorem}[Lucas's Theorem \cite{Lucas1878}]
Let $p$ be a prime number, and let $q$ and $r$ two integers such that
\begin{align*}
    q & = q_\ell p^\ell + q_{\ell-1} p^{\ell-1} + \cdots + q_1 p + q_0, \\
    r & = r_\ell p^\ell + r_{\ell-1} p^{\ell-1} + \cdots + r_1 p + r_0 ,
\end{align*}
where $q_i$ and $r_i$ are not necessarily nonzero for all $i\in \{ 0,1,\ldots,\ell \}$. Then the following congruence holds:
\[ \binom{q}{r} \equiv \prod_{i=0}^\ell \binom{q_i}{r_i} \bmod p. \]
\end{theorem}

\begin{example}
Consider $p=3$,  $q=15 $ and $r=7$. The $p$-ary representations of $q$ are $r $ are:
    \begin{align*}
        17&=1\cdot 3^2+2\cdot 3+2\cdot 3^0, \\
        7 &=0\cdot 3^2+2\cdot 3+1\cdot 3^0.
    \end{align*}
Therefore, according to Lucas's Theorem:
$$
\binom{17}{7}\equiv \binom{1}{0}\cdot \binom{2}{2}\cdot \binom{2}{1}\bmod 3 \equiv 2 \bmod 2.
$$
Notice that $\binom{17}{7}=19448$, which is congruent with 2 modulo 3. 
\end{example}

\begin{definition} Let $\{a_n\} $ and $\{b_n\} $  be two $p$-ary sequences.
We define the \textbf{Hadamard product} (\textbf{component-wise product}) of $\{ a_n \} $ and $\{ b_n \} $ as the sequence $\{ a_n \}  \odot \{b_n\}  = \{ a_n\cdot b_n \}$.

%\begin{align*}
%(a_0, a_1, a_2, \ldots, a_{n-1}) \odot (b_0, b_1, b_2, \ldots, b_{n-1}) = (a_0b_0, a_1b_1, a_2b_2, \ldots, a_{n-1}b_{n-1}).
%\end{align*}
\end{definition}

Consider a sequence $\{ a_n \} $ of period $T_1$ and a sequence $\{ b_n \} $ of period $T_2$. Denote by $\{a_0, a_1, a_2, \ldots, a_{T_1-1}\}$ and $\{b_0, b_1, b_2, \ldots, b_{T_2-1}\}$ its first $T_1$ and $T_2$ terms. Take $T' = \lcm(T_1,T_2)$, and define
 \begin{align*}
     \{c_0,c_1,c_2,\ldots,c_{T'-1}\} & = \underbrace{\{a_0,a_1,\ldots,a_{T_1-1},\ a_0,a_1,\ldots,a_{T_1-1},\ \ldots,\ a_0,a_1,\ldots,a_{T_1-1}\}}_{\frac{T'}{T_1} \text{ times}}, \\
   \{d_0,d_1,d_2,\ldots,d_{T'-1}\} & = \underbrace{\{b_0,b_1,\ldots,b_{T_2-1},\ b_0,b_1,\ldots,b_{T_2-1},\ \ldots,\ b_0,b_1,\ldots,b_{T_2-1}\}}_{\frac{T'}{T_2} \text{ times}}.
 \end{align*}
Then the first $T'$ elements of $\{a_n\}  \odot \{ b_n \} $ are given by $\{c_0d_0, c_1d_1, c_2d_2, \ldots, c_{T'-1}d_{T'-1}\}$, and its period $T$ divides $T'$.

\begin{example}
 Consider the binomial sequences over $\mathbb{F}_3$ (see Table~\ref{tab:2}).
 The Hadamard product of $\textbinomseq{n}{4}$ and  $\textbinomseq{n}{7}$ is the sequences  given by:
 \begin{align*}
    \binomseq{n}{4} \odot \binomseq{n}{7} = 
    \{ 000012021  \} \odot \{ 000000012 \} =
    \{ 000000022 \}.
    %= 2 \binomseq{n}{7}
 \end{align*}
\end{example}
%We denote
% \[
% (a_0,a_1,\ldots,a_{T_1-1}) \odot (b_0,b_1,\ldots,b_{T_2-1}) = (c_0d_0, c_1d_1, c_2d_2, \ldots, c_{T'-1}d_{T'-1}).
% \]
% \begin{definition} Let $(a_0,a_1,\ldots,a_n)$ and $(b_0,b_1,\ldots,b_m)$ be two vectors. We define the \textbf{Kronecker product} of these vectors as
% \[ (a_0,a_1,\ldots,a_n) \otimes (b_0,b_1,\ldots,b_m) = (a_0b_0,a_0b_1,\ldots,a_0b_m,a_1b_0,a_1b_1,\ldots,a_1b_m,\ldots,a_nb_0,a_nb_1,\ldots,a_nb_m). \]
% This definition can be easily extended to the Kronecker product of a sequences by a vector of finite size. Let $\{ a_0 \} $ be a sequence and $(b_0, b_1, b_2, \ldots, b_m)$ be a vector of size $m+1$. We define the \textbf{Kronecker product} of $\{a_n\} $ and $(b_0, b_1, b_2, \ldots, b_m)$ as the sequence given by
% \begin{align*}
% \{ a_n \}  & \otimes (b_0, b_1, b_2, \ldots, b_m) \\
% & = \{ (a_0) \odot (b_0, b_1, b_2, \ldots, b_m), (a_1) \odot  (b_0, b_1, b_2, \ldots, b_m), (a_2) \odot  (b_0, b_1, b_2, \ldots, b_m), \ldots \} \\
% & = \{ a_0b_0, a_0b_1, a_0b_2, \ldots, a_0b_m, a_1b_0, a_1b_1, a_1b_2, \ldots, a_1b_m, a_2b_0, a_2b_1, a_2b_2, \ldots, a_2b_m, \ldots \}.
% \end{align*}
% \end{definition}
 
 \begin{definition}  Consider the sequence $\{ a_n \} $ and the terms   $\{b_0, b_1, b_2, \ldots, b_m\}$. We define the \textbf{Kronecker product} of $\{a_n\} $ and $\{b_0, b_1, b_2, \ldots, b_m\}$ as the sequence given by
\begin{align*}
\{ a_n \}  & \otimes \{b_0, b_1, b_2, \ldots, b_m\} \\
& =   \{ a_0b_0, a_0b_1, a_0b_2, \ldots, a_0b_m, a_1b_0, a_1b_1, a_1b_2, \ldots, a_1b_m, a_2b_0, a_2b_1, a_2b_2, \ldots, a_2b_m, \ldots \}.
\end{align*}
\end{definition}

\begin{remark} In order to avoid the use of multiple parentheses, we consider that the Kronecker product precedes the Hadamard product. That is, in the expression $\{ a_n \} \otimes \{ b_n \}_{n \geq 0}^m \odot \{ c_n \}$, we first solve the Kronecker product and then the Hadamard product.
\end{remark}

\begin{definition}
Given a prime number $p$, we define the \textbf{basic $p$-expansion block} $B_p$ as the sequence
\[ B_p = \underbrace{\{11\ldots1\}}_{p \text{ terms}}. \]
Moreover, for $\ell\geq 1$, we have that
\[ B_{p^\ell} = B_{p^{\ell-1}} \otimes B_p = \underbrace{\{11\ldots1\}}_{\text{$p^\ell$ terms}}. \]
\end{definition}

Computing the Kronecker product of a sequence $\{ a_n \} $ by the basic $p$-expansion block $B_p$ produces an expanded version of the original sequence, where each term appears $p$ times consecutively:
\[
\{ a_n \}  \otimes 
\underbrace{\{11\ldots1\}}_{p \text{ terms}}
= 
\{ \underbrace{a_0,a_0,\ldots,a_0}_{p \text{ terms}},
\underbrace{a_1,a_1,\ldots,a_1}_{p \text{ terms}},
\underbrace{a_2,a_2,\ldots,a_2}_{p \text{ terms}},
\ldots \}.
\]

In the following lemma we show that the $p$-repeating expansion of some binomial sequence modulo $p$ is another binomial sequence modulo $p$. Moreover, we know precisely which sequence is.

\begin{lemma}\label{lem:SixB=Sip} Let $p$ be a prime number,  $k$ a positive integer  and $B_p$ the basic $p$-expansion block. Then, 
\[ \binomseq{n}{k} \otimes B_p \equiv \binomseq{n}{kp} \bmod p. \]
\end{lemma}

\begin{proof} Observe that
\[
\binomseq{n}{k} \otimes B_p =
\bigg\{
\underbrace{\binom{0}{k},\ldots,\binom{0}{k}}_{\text{$p$ terms}},
\underbrace{\binom{1}{k},\ldots,\binom{1}{k}}_{\text{$p$ terms}},
\underbrace{\binom{2}{k},\ldots,\binom{2}{k}}_{\text{$p$ terms}},
\ldots 
\bigg\}.
\]
We want to prove that 
\begin{align*}
& \bigg\{
\resizebox{0.6\linewidth}{!}{$
    \underbrace{\binom{0}{k},\binom{0}{k},\ldots,\binom{0}{k}}_{\text{$p$ terms}},
    \underbrace{\binom{1}{k},\binom{1}{k},\ldots,\binom{1}{k}}_{\text{$p$ terms}},
    \underbrace{\binom{2}{k},\binom{2}{k},\ldots,\binom{2}{k}}_{\text{$p$ terms}},
    \ldots 
$}
\bigg\}
\\
& \equiv
\bigg\{
\resizebox{0.8\linewidth}{!}{$
    \underbrace{\binom{ 0}{kp},\binom{   1}{kp},\ldots,\binom{ p-1}{kp}}_{\text{$p$ terms}},
    \underbrace{\binom{ p}{kp},\binom{ p+1}{kp},\ldots,\binom{2p-1}{kp}}_{\text{$p$ terms}},
    \underbrace{\binom{2p}{kp},\binom{2p+1}{kp},\ldots,\binom{3p-1}{kp}}_{\text{$p$ terms}},
    \ldots 
$}
\bigg\} \bmod p.
\end{align*}
 
That is, we need to prove that for all $j \in \{ 0,1,2,\ldots,p-1 \}$ the following congruence holds:
\[
\binom{np+j}{kp} \equiv \binom{n}{k} \bmod p.
\]

Consider the $p$-ary expansions of $n$ and $k$, i.e.,
\begin{align*}
    n & = n_\ell p^\ell + n_{\ell-1} p^{\ell-1} + \cdots + n_1 p + n_0, \\
    k & = k_\ell p^\ell + k_{\ell-1} p^{\ell-1} + \cdots + k_1 p + k_0. 
\end{align*}
with $n_i, k_i \in \{ 0,1,\ldots,p-1 \}$ and either $n_{\ell} \neq 0$ or $k_\ell \neq 0$. Then, $np+j$ and $kp$ are 
\begin{align*}
    np + j & = n_\ell p^{\ell+1} + n_{\ell-1} p^\ell + \cdots + n_1 p^2 + n_0 p + j, \\
    kp & = k_\ell p^{\ell+1} + k_{\ell-1} p^\ell + \cdots + k_1 p^2 + k_0 p + 0,
\end{align*}
so  Lucas's theorem implies the congruences
\[
\binom{np+j}{kp}
\equiv \prod_{i=0}^\ell \binom{n_i}{k_i} \cdot \binom{j}{0} \equiv \prod_{i=0}^\ell \binom{n_i}{k_i} \equiv \binom{n}{k} \bmod p.
\]    
\end{proof}
 
\begin{example}
Consider again the binomial sequences over $\mathbb{F}_3$.
The sequence $\textbinomseq{n}{18}=\textbinomseq{n}{6\cdot 3}$ can be generated through the Kronecker product of $\textbinomseq{n}{6}$ and $B_3=\{111\}$.
Let us generate the first period of $\textbinomseq{n}{18}$:
    \begin{align*}
           \binomseq{n}{18} =\binomseq{n}{6} \otimes B_3 & =  \{ 000000111 \}  \otimes \{1 1 1\}  \\
           &=
    \{ 000  |   000   |   000   |   000   |   000  |   000  |   111  |   111  |   111 \}.
    \end{align*}
\end{example}
Notice that there is no restriction for $k$ in Theorem~\ref{lem:SixB=Sip}, i.e., $k$ can also be a multiple of $p$.

The next lemma shows that in some cases the Hadamard product of two binomial sequences is again a binomial sequence.

\begin{lemma}\label{lem:Sip*Sj=Sip+j} Let $p^\ell$ be the power of a prime number, $m$ be an positive integer and $j \in \{ 0,1,\ldots,p^\ell-1 \}$. Then
%\[ S_{ip^m} \odot S_j = S_{ip^m+j}. \]
\[ \binomseq{n}{mp^\ell+j}  = \binomseq{n}{m p^\ell} \odot \binomseq{n}{j}. \]
\end{lemma}
\begin{proof} 
Notice that the right-hand side of the equality is
\[
\binomseq{n}{m p^\ell} \odot \binomseq{n}{j} = \left\{\!\binom{n}{mp^\ell} \binom{n}{j}\!\right\},
\]
so, we need to prove that for all $n \geq 0$ the following congruence holds: \begin{equation}\label{eq:congruence_to_prove_Sip*Sj=Sip+j}
\binom{n}{mp^\ell+j} \equiv \binom{n}{mp^\ell} \binom{n}{j} \bmod p.
\end{equation}
Let us consider the $p$-ary expansions of $m$ and $j$, i.e.,
\begin{align*}
    m & = m_r p^r + m_{r-1} p^{r-1} + \cdots + m_1 p + m_0, \\
    j & = j_{\ell-1} p^{\ell-1} + j_{\ell-2} p^{\ell-2} + \cdots + j_1 p + j_0.
\end{align*}
Thus, the $p$-ary expansion of $mp^\ell + j$ is
\begin{align*}
    mp^\ell + j & = m_r p^{\ell+r} + m_{r-1} p^{\ell+r-1} + \cdots + m_1 p^{\ell+1} + m_0 p^\ell + j_{\ell-1} p^{\ell-1} +
    %j_{\ell-2} p^{\ell-2} + 
    \cdots + j_1 p + j_0,
\end{align*}
We have to study two cases: first, we consider the case that $n < p^{\ell+r+1}$. We can write
\[ n = n_{\ell+r} p^{\ell+r} + n_{\ell+r-1} p^{\ell+r-1} + \cdots + n_1 p + n_0,
\]
with the $n_i$ not necessarily different from $0$, for all $i \in \{ 0,1,\ldots,\ell+r \}$, so in this case Lucas's theorem implies the congruence
\begin{equation}\label{eq:congruence_to_prove_Sip*Sj=Sip+j_1}
    \binom{n}{mp^\ell+j} \equiv \prod_{i=\ell}^{\ell+r} \binom{n_i}{m_{i-\ell}} \cdot \prod_{i=0}^{\ell-1} \binom{n_\ell}{j_\ell} \bmod p.
\end{equation}
Also,  Lucas's theorem implies the congruences
\begin{equation}\label{eq:congruence_to_prove_Sip*Sj=Sip+j_2}
    \binom{n}{mp^\ell} \equiv \prod_{i=\ell}^{\ell+r} \binom{n_i}{m_{i-\ell}} \cdot \prod_{i=0}^{\ell-1} \binom{n_i}{0} \equiv \prod_{i=\ell}^{\ell+r} \binom{n_i}{m_{i-\ell}} \bmod p,
\end{equation}
and
\begin{equation}\label{eq:congruence_to_prove_Sip*Sj=Sip+j_3}
    \binom{n}{j} \equiv \prod_{i=\ell}^{\ell+r} \binom{n_i}{0} \cdot \prod_{i=0}^{\ell-1} \binom{n_i}{j_i} \equiv \prod_{i=0}^{\ell-1} \binom{n_i}{j_i} \bmod p. 
\end{equation}
Therefore, congruences \eqref{eq:congruence_to_prove_Sip*Sj=Sip+j_1}, \eqref{eq:congruence_to_prove_Sip*Sj=Sip+j_2} and \eqref{eq:congruence_to_prove_Sip*Sj=Sip+j_3} imply congruence \eqref{eq:congruence_to_prove_Sip*Sj=Sip+j}.

Consider now the case $n \geq p^{r+\ell+1}$. Then, there is an integer $s \geq r+\ell+1$ such that the $p$-ary expansion of $n$ is of the form
\[ n = n_s p^s + \cdots + n_{\ell+r+1} p^{\ell+r+1} + n_{\ell+r} p^{\ell+r} + n_{\ell+r-1} p^{\ell+r-1} + \cdots + n_1 p + n_0, \]
with $n_s \neq 0$. Then,  Lucas's theorem implies the congruences
\begin{align}
    \label{eq:congruence_to_prove_Sip*Sj=Sip+j_4}
    \binom{n}{mp^\ell+j} & \equiv \prod_{i=\ell+r+1}^s \binom{n_i}{0} \prod_{i=\ell}^{\ell+r} \binom{n_i}{m_{i-\ell}} \cdot \prod_{i=0}^{\ell-1} \binom{n_\ell}{j_\ell} \equiv \prod_{i=\ell}^{\ell+r} \binom{n_i}{m_{i-\ell}} \cdot \prod_{i=0}^{\ell-1} \binom{n_\ell}{j_\ell} \bmod p. \\
    \label{eq:congruence_to_prove_Sip*Sj=Sip+j_5}
    \binom{n}{mp^\ell} & \equiv \prod_{i=\ell+r+1}^s \binom{n_i}{0} \prod_{i=\ell}^{\ell+r} \binom{n_i}{m_{i-\ell}} \cdot \prod_{i=0}^{\ell-1} \binom{n_i}{0} \equiv \prod_{i=\ell}^{\ell+r} \binom{n_i}{m_{i-\ell}} \bmod p, \\
    \label{eq:congruence_to_prove_Sip*Sj=Sip+j_6}
    \binom{n}{j} & \equiv \prod_{i=\ell+r+1}^s \binom{n_i}{0} \prod_{i=\ell}^{\ell+r} \binom{n_i}{0} \cdot \prod_{i=0}^{\ell-1} \binom{n_i}{j_i} \equiv \prod_{i=0}^{\ell-1} \binom{n_i}{j_i} \bmod p. 
\end{align}
Hence, congruences \eqref{eq:congruence_to_prove_Sip*Sj=Sip+j_4}, \eqref{eq:congruence_to_prove_Sip*Sj=Sip+j_5} and \eqref{eq:congruence_to_prove_Sip*Sj=Sip+j_6} imply congruence \eqref{eq:congruence_to_prove_Sip*Sj=Sip+j}, as we wanted to show.
\end{proof}
 
\begin{example}
Consider again the field $\mathbb{F}_3$.
The sequence
$\textbinomseq{n}{24} = \textbinomseq{n}{2\cdot 3^2+6}$ can be computed through the Hadamard product of $\textbinomseq{n}{18}$ and $\textbinomseq{n}{6}$. Let us compute the first period of $\textbinomseq{n}{24} = \textbinomseq{n}{2\cdot 3^2+6}$:
\begin{align*}   
    \binomseq{n}{24} = & \binomseq{n}{18} \odot \binomseq{n}{6} \\
=&\ \{ 000000000000000000111111111 
\} \odot
\{000000111000000111000000111 \} \\
 =&\ \{000000000000000000000000111 \}.
    \end{align*}
\end{example}
 
As a consequence of Lemmas \ref{lem:SixB=Sip} and \ref{lem:Sip*Sj=Sip+j}, we can deduce the following decomposition of the binomial sequences.

\begin{theorem}\label{thm:char_binomial_sequences} Let $k = k_\ell p^\ell + k_{\ell-1} p^{\ell-1} + \cdots + k_1 p + k_0$ be the $p$-ary expansion of $k$, with $k_i \in \{ 0,1,\ldots,p-1\}$. The $k$-th binomial sequence modulo $p$ is given by
%\[
%S_i = S_{i_k} \otimes B_p^k \odot S_{i_{k-1}} \otimes B_p^{k-1} \odot \cdots \odot S_{i_1} \otimes B_p \odot S_{i_0}.
%\]
\[
\binomseq{n}{k} = 
\binomseq{n}{k_\ell} \otimes B_{p^\ell} \odot
\binomseq{n}{k_{\ell-1}} \otimes B_{p^{\ell-1}} \odot
\cdots \odot 
\binomseq{n}{k_1} \otimes B_{p} \odot
\binomseq{n}{k_0} \otimes B_{p^0}.
\]
\end{theorem}

\begin{proof} %Consequence of Lemma \ref{lem:SixB=Sip} and Lemma \ref{lem:Sip*Sj=Sip+j}. 
On one hand, Lemma~\ref{lem:Sip*Sj=Sip+j} implies that 
\begin{align*}  
\binomseq{n}{k} = & \binomseq{n}{k_\ell p^\ell + k_{\ell-1} p^{\ell-1} + \cdots + k_1 p + k_0}\\
= &
\binomseq{n}{k_\ell p^\ell} \odot \binomseq{n}{k_{\ell-1} p^{\ell-1} + \cdots + k_1 p + k_0}.
\end{align*}
On the other hand, we can apply Lemma~\ref{lem:SixB=Sip} repeatedly to deduce that
\[
\binomseq{n}{k_\ell p^\ell} = \binomseq{n}{k_\ell p^{\ell-1}} \otimes B_p 
%= \binomseq{n}{k_\ell p^{\ell-2}} \otimes B_p \otimes B_p 
= \cdots = \binomseq{n}{k_\ell} \otimes \underbrace{B_p \otimes \cdots \otimes B_p}_{\ell \text{ times}} = \binomseq{n}{k_\ell} \otimes B_{p^\ell},
\]
so from both equalities we conclude that
\[
\binomseq{n}{k} = 
\binomseq{n}{k_\ell} \otimes B_{p^\ell} \odot \binomseq{n}{k_{\ell-1} p^{\ell-1} + \cdots + k_1 p + k_0}.
\]
Reiterating the process with $\textbinomseq{n}{k_{\ell-1} p^{\ell-1} + \cdots + k_1 p + k_0}$ yields the result.
\end{proof}

From now on, we suppress the last term $B_{p^0}$, since any sequence, when multiplied by it via the Kronecker product, remains unchanged.
 
As  a consequence of Theorem~\ref{thm:char_binomial_sequences}, any binomial sequence can be constructed from the first $p$ binomial sequences $\textbinomseq{n}{t}$, for $t=0,1,\ldots,p-1$ called \textbf{$p$-ary basic binomial sequences}.
% \begin{align*}
%     \binomseq{n}{0} & = \left\lbrace\!\binom{0}{0},\binom{1}{0},\binom{2}{0},\binom{3}{0},\ldots \right\} = \{ 1\ 1\ 1\ 1\ \ldots\,\}, \\
%     \binomseq{n}{1} & = \left\lbrace\!\binom{0}{1},\binom{1}{1},\binom{2}{1},\binom{3}{1},\ldots \right\} = \{ 0\ 1\ 2\ 3\ \ldots\,\}, \\
%     & \ \vdots \\
%     \binomseq{n}{p-1} & = \left\lbrace\!\binom{0}{p-1},\binom{1}{p-1},\binom{2}{p-1},\binom{3}{p-1},\ldots \right\}. \\
% \end{align*}
%We call these  sequences the \textbf{$p$-ary basic binomial sequences}. Then, Theorem \ref{thm:char_binomial_sequences} states 
Therefore, any binomial sequence can be generated by computing the Hadamard and Kronecker products of these basic binomial sequences and the basic $p$-expansion block.

%\begin{corollary} The $i$-th binomial sequence modulo $p$ is of period $p^{\left\lfloor \log_p (i) \right\rfloor + 1}$.    
%\end{corollary}

\begin{example} Over $\mathbb{F}_3$, all the binomial sequences can be generated with $\textbinomseq{n}{0}$, $\textbinomseq{n}{1}$ and $\textbinomseq{n}{2}$. For instance, the first binomial sequences are generated as:
\begin{align*}
 3=1\cdot 3^1+0\cdot 3^0 \Rightarrow \textbinomseq{n}{3} & =  \textbinomseq{n}{1} \otimes B_3 \odot \textbinomseq{n}{0} 
\\
& = \textstyle\{0 1 2\ldots \}\otimes \{1 1 1\}\odot \textstyle\{111\ldots \}=
\{0 0 0 1 1 1 2 2 2 \ldots \}.  
\\\\
 4=1\cdot 3^1+1\cdot 3^0 \Rightarrow \textbinomseq{n}{4} & = \textbinomseq{n}{1}  \otimes B_3 \odot \textbinomseq{n}{1}
\\
& = \{ 0 1 2 \ldots \} \otimes \{ 1 1 1 \} \odot \{ 0 1 2 \ldots \} = \{ 0 0 0 1 1 1 2 2 2 \ldots \} \odot \{ 0 1 2 \ldots \}
\\
& =  \{ 0 0 0 0 1 2 0 2 1 \ldots \}.
\\\\
 5=1\cdot 3^1+2\cdot 3^0 \Rightarrow \textbinomseq{n}{5} & = \textbinomseq{n}{1}  \otimes B_3 \odot \textbinomseq{n}{2} 
\\
& = \{ 0 1 2 \ldots \} \otimes \{ 1 1 1 \} \odot \{ 0 0 1 \ldots \} = \{ 0 0 0 1 1 1 2 2 2 \ldots \} \odot \{ 0 0 1 \ldots \}
\\
& =  \{ 0 0 0 0 0 1 0 0 2 \ldots \}.
\\\\
 6=2\cdot 3^1+0\cdot 3^0 \Rightarrow \textbinomseq{n}{6} & = \textbinomseq{n}{2}  \otimes B_3
\\
& = \{ 0 0 1 \ldots \} \otimes \{ 1 1 1 \} = \{ 0 0 0 0 0 0 1 1 1 \ldots \}
\\\\
 7=2\cdot 3^1+1\cdot 3^0 \Rightarrow \textbinomseq{n}{7} & = \textbinomseq{n}{2}  \otimes B_3 \odot \textbinomseq{n}{1} 
\\
& = \{ 0 0 1 \ldots \} \otimes \{ 1 1 1 \} \odot \{ 0 1 2 \ldots \} = \{ 0 0 0 0 0 0 1 1 1 \ldots \} \odot \{ 0 1 2 \ldots \} 
\\
& =  \{ 0 0 0 0 0 0 0 1 2 \ldots \}. 
\\\\
 8=2\cdot 3^1+2\cdot 3^0 \Rightarrow \textbinomseq{n}{8} & =\textbinomseq{n}{2} \otimes B_3 \odot \textbinomseq{n}{2} 
\\
& = \{ 0 0 1 \ldots \} \otimes \{ 1 1 1 \} \odot \{ 0 0 1 \ldots \} = \{ 0 0 0 0 0 0 1 1 1 \ldots \} \odot \{ 0 0 1 \ldots \} 
\\
& = \{ 0 0 0 0 0 0 0 0 1 \ldots \}.
\\
\end{align*}
\end{example}

%\begin{remark}
Using the decomposition of Theorem \ref{thm:char_binomial_sequences}, we can compute easily the binomial sequences of the form $\big\{\!\binom{n}{p^L}\!\big\}$, since $\textbinomseq{n}{p^L} = \textbinomseq{n}{1} \otimes B_{p^L}$, so its first period is of the form
    \[ \{\underbrace{0\ 0\ \ldots\ 0}_{p^L \text{ times}}\ \underbrace{1\ 1\ \ldots\ 1}_{p^L \text{ times}} \ \underbrace{2\ 2\ \ldots\ 2}_{p^L \text{ times}}\ \underbrace{3\ 3\ \ldots\ 3}_{p^L \text{ times}} \ldots \  \underbrace{p-1\ p-1\ \ldots\ p-1}_{p^L \text{ times}}\}. \]
% \end{remark}

Another immediate consequence of Theorem~\ref{thm:char_binomial_sequences}, together with the previous remark, are the results given in the remarks of subsection~\ref{sub:seq3_5}, %Theorems~\ref{th:bin_3}, \ref{th:bin_5}, 
and \ref{th:bin_p}. Note that Theorem~\ref{th:bin_p} can be reformulated in terms of Theorem~\ref{thm:char_binomial_sequences}, as follows. 
%  \begin{theorem} 

 Let $L, r$ be non-negative integers. 
 For $j \in \{1,2,\ldots, p-1\}$, consider $\textbinomseq{n}{j} = (s^{(j)}_0,s^{(j)}_1,\ldots, s^{(j)}_{p-1})$. We have that:
% The sequences $\textbinomseq{n}{k}$, where $k=p^L+r$ and $0 \leq r<(p-1)\cdot p^L$ have the following forms:
\bigskip
  \begin{enumerate}[(a),topsep=0pt,partopsep=0pt,parsep=0pt,itemsep=0pt]
      \item For the multiples of the power of $p$, 
  \begin{equation*}
  \binomseq{n}{j \cdot p^L}=\left\{\underbrace{s^{(j)}_0\ldots s^{(j)}_0}_{p^{L}}, \underbrace{s^{(j)}_1\ldots s^{(j)}_1}_{p^{L}}, \ldots, \underbrace{s^{(j)}_{p-1}\ldots s^{(j)}_{p-1}}_{p^{L}}\right\} 
  =
  \binomseq{n}{j } \otimes B_{p^L}.
  \end{equation*}
 % for any $j \in \{1,2,\ldots, p-1\}$.
  \item For $i=1,2, \ldots L$ and $p^{i-1}\leq r<p^{i}$,
  \begin{align*}
  \binomseq{n}{\textcolor{blue}{j}\cdot p^L+r}
  & =
  \left\{\underbrace{\textcolor{blue}{s^{(j)}_0} \cdot \binomseq{n}{r}}_{p^{L-i}}, \underbrace{\textcolor{blue}{s^{(j)}_1} \cdot \binomseq{n}{r}}_{p^{L-i}}, \ldots, \underbrace{\textcolor{blue}{s^{(j)}_{p-1}} \cdot \binomseq{n}{r}}_{p^{L-i}}\right\} \\
  & = 
  \binomseq{n}{j} \otimes B_{p^L} \odot \binomseq{n}{r}.
   % \text{ if } \ 5^{i-1}\leq r<5^{i}  \text{ for } i=1,2, \ldots L\\
   \end{align*}
 %for any $j \in \{1,2,\ldots, p-1\}$.
    \end{enumerate}
%\end{theorem}

We conclude this section with a result which shows how to write a sequence with period $p^L$ as a combination of the first $p$ binomial sequences.
\begin{theorem}
    Let $\{s_n\}$ be a sequence with binomial representation $\sum_{k=0}^{p^L-1}c_k \textbinomseq{n}{k}$, and let $k=k_\ell p^\ell+ \cdots+ k_1 p+k_0p^0$ be the $p$-ary representation  of $k$.
    The sequence $\{s_n\}$ can be written as:
    \begin{align*}
       \sum_{k=0}^{p^L-1}c_k\binomseq{n}{k}&=
     \sum_{k=0}^{p^L-1}
     c_k\left(
    \bigodot_{i=0}^{i=\ell}
    \binomseq{n}{k_i}\otimes B_{p^i}\right).\\ 
  %  &=\bigodot_{k=0}^{p^l-1 }\sum_{i=0}^{p-1}c_k\binomseq{n}{k_{i}}\otimes B_{p^i}
    \end{align*}
%  where the $p$-ary expansion of $k=0,1, \ldots, p-1$ is given by $k=k_{p^\ell}p^{\ell}+\ldots + j_{1}p+j_0$.
\end{theorem}

\section{A Horner-like algorithm for binomial sequences}

In this section, we introduce an algorithm, based on the result given in Theorem~\ref{thm:char_binomial_sequences}, to compute any binomial sequence over $\mathbb{F}_p$ as a combination of the first $p$ binomial sequences, using the Hadamard product, the Kronecker product, and a basic block expansion.
This algorithm makes use of Horner's rule, which we will explain in the following.

Horner's rule~\cite{Horner1819} (see also \cite[pag. 486]{Knuth1997}) is an algorithm to evaluate an univariate polynomial
\[ p(x) = a_\ell x^\ell + a_{\ell-1} x^{\ell-1} + \cdots + a_1 x + a_0, \ \ \ a_\ell \neq 0, \]
efficiently. It is based on the idea of rewriting this polynomial as
\[ p(x) = \left( \left( (a_\ell x + a_{\ell-1}) x + \cdots \right) x + a_1 \right) x + a_0. \]
Horner's rule works as follows:
\begin{center}
    \begin{minipage}{0.9\textwidth}
        \begin{algorithm}[H]
            \caption{Horner's rule}\label{alg:three2}
            \KwData{A polynomial $p(x) = a_\ell x^\ell + \cdots + a_1 x + a_0$ and a point $A$ to evaluate $p(x)$}
            \KwResult{The evaluation $p(A)$}
            $R \gets a_\ell$\;
            \For{$i \gets \ell-1$ \KwTo $0$}{
                $R \gets R \cdot A$ \\                
                $R \gets R + a_i$
            }
            \Return $R$\;
        \end{algorithm}
    \end{minipage}
\end{center}

Evaluating $p(x)$ using Horner's rule requires only $\ell$ additions and $\ell \in \mathcal{O}(\ell)$ multiplications instead of the $\sum_{i=1}^\ell i = \frac{\ell(\ell+1)}{2} \in \mathcal{O}(\ell^2)$ multiplications we need if each monomial of $p(x)$ is evaluated individually and powers are computed by repeated multiplication. Notice that this is the same principle used in the binary exponentiation, also known as \emph{square-and-multiply} method  \cite[Section 14.6]{Menezes1996bk}
%\cite[pag. 614]{Knuth1997}\todomiguel{Mirar cita Menezes en qué apéndice está y quitar esta referencia}: 
to compute $a^k$ we consider the binary expansion of $k$
\[ k = k_\ell \cdot 2^\ell + k_{\ell-1} \cdot 2^{\ell-1} + \cdots + k_1 \cdot 2 + k_0, \ \ \ k
_\ell \neq 0. \]
This binary expansion can be rewritten as
\[ k = \left( \left( (k_\ell \cdot 2 + k_{\ell-1}) \cdot 2 + \cdots \right) \cdot 2 + k_1 \right) \cdot 2 + k_0. \]
From this expression one can deduce the following algorithm
\begin{center}
    \begin{minipage}{0.9\textwidth}
        \begin{algorithm}[H]
            \caption{Binary exponentiation}\label{alg:three3}
            \KwData{A value $a$ and an integer $k \geq 0$.}
            \KwResult{The power $a^k$}
            Compute the binary expansion of $k = (k_\ell,k_{\ell-1},\ldots,k_1,k_0)_2$, $k_\ell \neq 0$\;
            $R \gets a^{k_\ell}$\;
            \For{$i \gets \ell-1$ \KwTo $0$}{
                $R \gets R^2$\;
                $R \gets R \cdot a^{k_i}$\;
            }
            \Return $R$\;
        \end{algorithm}
    \end{minipage}
\end{center}

\noindent We follow the same idea. Theorem \ref{thm:char_binomial_sequences} gives us a decomposition of a binomial sequence $\textbinomseq{n}{k}$ modulo $p$:
%\[
%S_i = S_{i_k} \otimes B_p^k \odot S_{i_{k-1}} \otimes B_p^{k-1} \odot \cdots \odot S_{i_1} \otimes B_p \odot S_{i_0}.
%\]
\[
\binomseq{n}{k} = 
\binomseq{n}{k_\ell} \otimes B_p^\ell \odot
\binomseq{n}{k_{\ell-1}} \otimes B_p^{\ell-1} \odot
\cdots \odot 
\binomseq{n}{k_1} \otimes B_p \odot
\binomseq{n}{k_0}.
\]
Properties of $\odot$ and $\otimes$ allows us to rewrite this decomposition in the same way as in the Horner's rule:
%\[
%S_i = ((((S_{i_k} \otimes B_p \odot S_{i_{k-1}}) \otimes B_p \odot S_{i_{k-2}}) \otimes B_p \cdots\ ) \otimes B_p \odot S_{i_1}) \otimes B_p \odot S_{i_0}.
%\]
\[
\binomseq{n}{k} 
=
\left(
    \left(
        \left(
            \binomseq{n}{k_\ell} \otimes B_p \odot
            \binomseq{n}{k_{\ell-1}}
        \right) \otimes B_p \odot
        \cdots
    \right) \otimes B_p \odot \binomseq{n}{k_1}
\right) \otimes B_p \odot \binomseq{n}{k_0}.
\]
Moreover, from this decomposition, we can have the following efficient Horner-like algorithm, that we call \emph{expand-and-multiply}, to compute the first period of the binomial sequence:
\begin{center}
    \begin{minipage}{0.9\textwidth}
        \begin{algorithm}[H]
            \caption{Algorithm to compute the $k$-th binomial sequence}\label{alg:three}
            \KwData{A prime number $p$ and a non-negative integer $k$}
            \KwResult{The $k$-th binomial sequence $\textbinomseq{n}{k}$ modulo $p$}
            Compute $S_0 \gets 1,\, S_1 \gets \textbinomseq{n}{1}_{n \geq 0}^{p-1},\, S_2 \gets \textbinomseq{n}{2}_{n \geq 0}^{p-1},\ldots,\, S_{p-1} \gets \textbinomseq{n}{p-1}_{n \geq 0}^{p-1}$\;
            Compute the $p$-ary expansion of $k = (k_\ell,k_{\ell-1},\ldots,k_1,k_0)_p$,$k_\ell \neq 0$ \;
            $R \gets S_{k_\ell}$\;
            \For{$i \gets \ell-1$ \KwTo $0$}{
                $R \gets S \otimes B_p$\; %\tcp*{only the first $p^i$ terms}
                $R \gets S \odot S_{k_i}$\; %\tcp*{only the first $p^{\ell-i}$ terms}
            }
            \Return $R$\;
        \end{algorithm}
    \end{minipage}
\end{center}

In step 1 we compute the basic sequences. These sequences may be precomputed to save computation time. We show an small example of this algorithm in Example \ref{ex:horner-like_computation}.

\begin{example}\label{ex:horner-like_computation}
Let $p = 5$. The first periods of the basic sequences are
\begin{align*}
    S_0 & = \{1\}, &
    S_1 & = \{01234\}, &
    S_2 & = \{00131\}, &
    S_3 & = \{00014\}, &
    S_4 & = \{00001\}.    
\end{align*}
Consider $k = 66$. This number has the $5$-ary expansion $66 = 2 \cdot 5^2 + 3 \cdot 5 + 1$, so $(k_2,k_1,k_0) = (2,3,1)$. We start by considering the sequence $R = S_2 = \{00131\}$.
\begin{itemize}
    \item $i=1$: we compute 
    \begin{align*}
        R  \leftarrow R \otimes B_5 & = \{00131\} \otimes \{11111\} = \{00000\ 00000\ 11111\ 33333\ 11111\}, \\
        \\
        R \leftarrow R \odot S_3 &  = \{00000\ 00000\ 11111\ 33333\ 11111\} \odot \{00014\} \\ 
        & = \{00000\ 00000\ 00014\ 00032\ \}.
    \end{align*}

    \item $i=0$: we compute 
    \begin{align*}
        R & \leftarrow R  \otimes B_5 \\
        & = \ \{0000000000000140003200014 \otimes 11111\} \\
        &= \ \{00000\ 00000\ 00000\ 00000\ 00000\ 00000\ 00000\ 00000\ 00000\ 00000\ 00000\ 00000 \} \\
        & \ \{00000\ 11111 \ 44444\ 00000\ 00000\ 00000\ 33333\ 22222\ 00000\ 00000\ 00000\ 11111\ 44444\}, \\
        \\
        R & \leftarrow R  \odot S_1 =\ \{00000\ 00000\ 00000\ 00000\ 00000\ 00000\ 00000\ 00000\ 00000\ 00000\} \\     
        &  \ \{00000\ 00000\ 11111 \ 44444\ 00000\ 00000\ 00000\ 33333\ 22222\ 00000\ 00000\ 00000\ 11111\}\\
        & \ \{44444\} \ \odot\ \{01234\} \\
        = &\ \{00000\ 00000\ 00000\ 00000\ 00000\ 00000\ 00000\ 00000\ 00000\ 00000\ 00000\ 00000\ 00000\} \\
        & \ \{01234 \ 04321\ 00000\ 00000\ 00000\ 03142\ 02413\ 00000\ 00000\ 00000\ 01234\ 04321\}. 
    \end{align*}
\end{itemize}
So the first period of the binomial sequence $\textbinomseq{n}{66}$ is
\begin{align*}
    R = &\ 
\{0000000000000000000000000\ 
0000000000000000000000000\ 
0000000000000000123404321\}\\
&\ 
\{0000000000000000314202413\ 
0000000000000000123404321\}.
\end{align*}
\end{example}

\section{Conclusions and future work}

In this work we have extended and improved the results in \cite{Cardell2019a} about binomial sequences from the binary field $\mathbb{F}_2$ to arbitrary prime fields $\mathbb{F}_p$. Moreover, we have proved some structural properties of periodic sequences in $\mathbb{F}_p$ of period $p^L$, we have established a connection between the characteristic polynomial and linear complexity of these sequences and those of the binomial sequences of its binomial decomposition and we have studied the action of the shift operator on these sequences.
We have proved that any binomial sequence can be obtained through the first $p$ binomial sequences. Moreover, we have presented an algorithm which allows us to compute any binomial sequence over $\mathbb{F}_p$ as a combination of these first $p$ binomial sequences, using the Hadamard product and the Kronecker product.  

It is known that binary binomial sequences can be generated by linear cellular automata. As a future, we aim to investigate whether similar structures over finite fields can produce $p$-ary binomial sequences, as well as the behavior of their linear combinations. Moreover, we intend to explore irregular decimation generators in the context of finite fields, rather than restricting our analysis to the binary case.

\bmhead{Acknowledgements}
 
The work of the first and third author was partially supported by the Spanish I+D+i project PID2022-142159OB-I00 of the Ministerio de Ciencia e Innovaci\'{o}n, I+D+i project
CIAICO/2022/167 of the Generalitat Valenciana, and the I+D+i project VIGROB23-287 and UADIF23-132 of the University of Alicante.  
The second author was supported by CNPq with process 405842/2023-6 and by FAPESP with process 2024/05051-7.

\section*{Declarations}
Not applicable

% Some journals require declarations to be submitted in a standardised format. Please check the Instructions for Authors of the journal to which you are submitting to see if you need to complete this section. If yes, your manuscript must contain the following sections under the heading `Declarations':

% \begin{itemize}
% \item Funding
% \item Conflict of interest/Competing interests (check journal-specific guidelines for which heading to use)
% \item Ethics approval and consent to participate
% \item Consent for publication
% \item Data availability 
% \item Materials availability
% \item Code availability 
% \item Author contribution
% \end{itemize}

% \noindent
% If any of the sections are not relevant to your manuscript, please include the heading and write `Not applicable' for that section. 

%%===================================================%%
%% For presentation purpose, we have included        %%
%% \bigskip command. Please ignore this.             %%
%%===================================================%%
\bigskip
\begin{flushleft}%
Editorial Policies for:

\bigskip\noindent
Springer journals and proceedings: \url{https://www.springer.com/gp/editorial-policies}

\bigskip\noindent
Nature Portfolio journals: \url{https://www.nature.com/nature-research/editorial-policies}

\bigskip\noindent
\textit{Scientific Reports}: \url{https://www.nature.com/srep/journal-policies/editorial-policies}

\bigskip\noindent
BMC journals: \url{https://www.biomedcentral.com/getpublished/editorial-policies}
\end{flushleft}

%%===========================================================================================%%
%% If you are submitting to one of the Nature Portfolio journals, using the eJP submission   %%
%% system, please include the references within the manuscript file itself. You may do this  %%
%% by copying the reference list from your .bbl file, paste it into the main manuscript .tex %%
%% file, and delete the associated \verb+\bibliography+ commands.                            %%
%%===========================================================================================%%

\bibliography{sn-bibliography}% common bib file

@BOOK{Lidl1986bk,
  Author    = {Lidl, Rudolf   and   Niederreiter, Harald},
  Title     = {Introduction to Finite Fields and Their Applications},
  Publisher = {Cambridge University Press},
  Address   = {New York, NY},
  Year      = 1986,
  ISBN      = {0-521-30706-6}
}

@ARTICLE{Cardell2019a,
  Author       = {Cardell, Sara~D. and F\'uster-Sabater, Amparo},
  Title        = {Binomial Representation of Cryptographic Binary Sequences and Its Relation to Cellular Automata},
  Journal      = {Complexity},
  Volume	   = {2019},
  Number       = {},
  ID           = {Article ID 2108014},
  Year         = {2019},
  Pages        = {1--13},
 DOI = {10.1155/2019/2108014}
 }

@article{CardellFuster2016,
    address={Philadelphia},
    title={Linear models for the self-shrinking generator based on {CA}},
    volume={11},
    ISSN={1557-5977},
    number={2-3},
    journal={Journal of Cellular Automata},
    publisher={Old City Publishing},
    author={Cardell, Sara D. and F\'{u}ster-Sabater, Amparo},
    year={2016},
    pages={195-211}
}

@book{GolombGong2005,
    title={Signal Design for Good Correlation: For Wireless Communication, Cryptography, and Radar},
    publisher={Cambridge University Press},
    author={Golomb, Solomon W. and Gong, Guang},
    year={2005},
    address={Cambridge},
    place={Cambridge}
}

@article{Bush2021,
   title={Period sets of linear recurrences over finite fields and related commutative rings},
   volume={14},
   ISSN={1944-4176},
   url={http://dx.doi.org/10.2140/involve.2021.14.361},
   DOI={10.2140/involve.2021.14.361},
   number={3},
   journal={Involve, a Journal of Mathematics},
   publisher={Mathematical Sciences Publishers},
   author={Bush, Michael R. and Quijada, Danjoseph},
   year={2021},
   month=jul, pages={361–376} }

@BOOK{Golomb1982bk,
  Author    = {Golomb, Solomon W.},
  Title     = {Shift Register-Sequences},
  Publisher = {Aegean Park Press},
  Address   = {Laguna Hill, California},
  Year      = 1982,
  ISBN      = {}
}

@article{Tsaban2002,
title = {Efficient Linear Feedback Shift Registers with Maximal Period},
journal = {Finite Fields and Their Applications},
volume = {8},
number = {2},
pages = {256-267},
year = {2002},
issn = {1071-5797},
doi = {https://doi.org/10.1006/ffta.2001.0339},
url = {https://www.sciencedirect.com/science/article/pii/S1071579701903399},
author = {Boaz Tsaban and Uzi Vishne}
}

@article{Goltvanitsa2012,
title = {Skew linear recurring sequences of maximal period over {G}alois rings. },
journal = {Journal of Mathematical Sciences },
volume = {187},
number = {},
pages = {115--128},
year = {2012},
doi = {https://doi.org/10.1007/s10958-012-1054-2},
author = {Goltvanitsa, M.A. and Zaitsev, S.N. and Nechaev, A.A.},
}

@article{GANESAN2021,
title = {Linear recurrences over a finite field with exactly two periods},
journal = {Advances in Applied Mathematics},
volume = {127},
pages = {102180},
year = {2021},
issn = {0196-8858},
doi = {https://doi.org/10.1016/j.aam.2021.102180},
url = {https://www.sciencedirect.com/science/article/pii/S019688582100018X},
author = {Ghurumuruhan Ganesan}
}

@ARTICLE{Cardell2020c,
  Author       = { Cardell, Sara D. and Climent, Joan-Josep  and F\'uster-Sabater, Amparo and Requena, Ver\'onica   },
  Title        = {Representations of Generalized Self-Shrunken Sequences},
  Journal      = {Mathematics},
  Volume	   = {},
  Number       = {8},
  ID           = {1460},
  Year         = {2020},
  Pages        = {1--26},
 DOI = {10.3390/math8061006}
 }

@PHDTHESIS{Bos2024PhDT,
  Author  = {Bos, Steven},
  Title   = {Beyond 0 and 1: A mixed radix design and verification workflow for modern ternary computers},
  School  = {Faculty of Technology, Natural Sciences and Maritime Studies,
University of South-Eastern Norway},
  Address = {},
  Month   = {},
  Year    = 2024
}

@MISC{Brousentsov,
author = {Brousentsov, N. P. and Maslov, S. P. and Ramil Alvarez, J. and Zhogolev, E.A.},
title = {Development of ternary computers at Moscow State University},
Howpublished ={},
month = {(Last accessed March 2025)},
year = {},
note = {https://www.computer-museum.ru/english/setun.htm}
}

@ARTICLE{Lucas1878,
  Author       = { Lucas, Edouard  },
  Title        = {Sur les congruences des nombres eulériens et des coefficients différentiels des fonctions trigonométriques suivant un module premier },
  Journal      = {Bulletin de la Société Mathématique de France},
  Volume	   = {6},
  Number       = {},
  Year         = {1878},
  Pages        = {49--54},
 DOI = {10.24033/bsmf.127}
 }

@BOOK{Menezes1996bk,
  Author    = {Menezes, Alfred J.   and   van Oorschot,  Paul C.   and   Vanstone, Scott A.},
  Title     = {Handbook of Applied Cryptography},
  Publisher = {CRC Press},
  Address   = {Boca Raton, FL},
  Year      = 1996,
  ISBN      = {0-8493-8523-7}
}

@BOOK{Horn2012bk,
  Author    = {Horn, Roger A. and  Johnson, Charles R.},
  Title     = {Matrix Analysis (2nd edition)},
  Publisher = {Cambridge University Press},
  Address   = {Cambridge},
  Year      = {2012},
  ISBN      = {9780511810817},
  DOI       = {https://doi.org/10.1017/CBO9780511810817}
}

@book{CoxLittleOShea2015,
author = {Cox, David A. and Little, John and O'Shea, Donal},
title = {Ideals, Varieties, and Algorithms: An Introduction to Computational Algebraic Geometry and Commutative Algebra, 3/e (Undergraduate Texts in Mathematics)},
year = {2015},
isbn = {978-3-319-16720-6},
publisher = {Springer-Verlag},
address = {Berlin, Heidelberg},
doi = {https://doi.org/10.1007/978-3-319-16721-3}
}

@book{HornJohnson1991,
    place={Cambridge},
    title={Topics in Matrix Analysis},
    publisher={Cambridge University Press},
    author={Horn, Roger A. and Johnson, Charles R.},
    year={1991},
    address={Cambridge}
}

@article{CallVellman1993,
    author = {Gregory S. Call and Daniel J. Velleman},
    title = {Pascal's Matrices},
    journal = {The American Mathematical Monthly},
    volume = {100},
    number = {4},
    pages = {372--376},
    year = {1993},
    publisher = {Taylor \& Francis},
    doi = {10.1080/00029890.1993.11990415},
}

@article{AbramsFishkingValdesLeon2000,
    author = {Lowell Abrams and Donniell E. Fishking and Silvia Valdes-Leon},
    title = {Reflecting the {P}ascal matrix about its main antidiagonal},
    journal = {Linear and Multilinear Algebra},
    volume = {47},
    number = {2},
    pages = {129--136},
    year = {2000},
    publisher = {Taylor \& Francis},
    doi = {10.1080/03081080008818638}
}

@article{Kubelka2002,
     title = {Decomposition of {P}ascal's kernels mod {$p^s$}},
     author = {Kubelka, Richard P.},
     journal = {Integers},
     volume = {2},
     year = {2002},
     zbl = {1107.11304},
     language = {en},
     url = {http://dml.mathdoc.fr/item/01896938},
     doi = {10.5281/zenodo.7589277}
}

@article{Horner1819,
 ISSN = {02610523},
 URL = {http://www.jstor.org/stable/107508},
 author = {W. G. Horner},
 journal = {Philosophical Transactions of the Royal Society of London},
 pages = {308--335},
 publisher = {Royal Society},
 title = {A New Method of Solving Numerical Equations of All Orders, by Continuous Approximation},
 volume = {109},
 year = {1819}
}

@book{Knuth1997,
  address = {Boston},
  author = {Knuth, Donald E.},
  edition = {Third},
  isbn = {0201896842 9780201896848},
  publisher = {Addison-Wesley},
  refid = {174763889},
  title = {The Art of Computer Programming, Volume 2: Seminumerical Algorithms},
  year = 1997
}

@INPROCEEDINGS{SmarandacheGluessingLuerssenRosenthal2002,
  author={Smarandache, R. and Gluesing-Luerssen, H. and Rosenthal, J.},
  booktitle={Proceedings IEEE International Symposium on Information Theory}, 
  title={Strongly {MDS} convolutional codes, a new class of codes with maximal decoding capability}, 
  year={2002},
  volume={},
  number={},
  pages={426-},
  doi={10.1109/ISIT.2002.1023698}}
%% if required, the content of .bbl file can be included here once bbl is generated
%%\input sn-article.bbl

\begin{appendices}

\section{Table}\label{secA1}
First 25 binomial sequences over $\mathbb{F}_5$, periods   and linear complexities. 
 $$
\def\arraystretch{1.1}
\begin{array}{|c|c|c|c|}\hline
\text{Binomial sequence}&\text{First terms} & \text{Period}& \text{LC}\\ \hline
\textbinomseq{n}{0} &  1\:\!1\:\!1\:\!1\:\!1\:\!1\:\!1\:\!1\:\!1\:\!1\:\!1\:\!1\:\!1\:\!1\:\!1\:\!1\:\!1\:\!1\:\!1\:\!1\:\!1\:\!1\:\!1\:\!1\:\!1 &  1 &  1 \\
\textbinomseq{n}{1} & 0\:\!1\:\!2\:\!3\:\!4\:\!0\:\!1\:\!2\:\!3\:\!4\:\!0\:\!1\:\!2\:\!3\:\!4\:\!0\:\!1\:\!2\:\!3\:\!4\:\!0\:\!1\:\!2\:\!3\:\!4 &  5 &  2 \\
\textbinomseq{n}{2} & 0\:\!0\:\!1\:\!3\:\!1\:\!0\:\!0\:\!1\:\!3\:\!1\:\!0\:\!0\:\!1\:\!3\:\!1\:\!0\:\!0\:\!1\:\!3\:\!1\:\!0\:\!0\:\!1\:\!3\:\!1
 &  5 &  3 \\
\textbinomseq{n}{3} & 0\:\!0\:\!0\:\!1\:\!4\:\!0\:\!0\:\!0 \:\!1\:\!4\:\!0\:\!0\:\!0\:\! 1\:\!4\:\!0\:\!0\:\!0 \:\!1\:\!4\:\!0\:\!0\:\!0\:\! 1\:\!4&  5 &  4 \\
\textbinomseq{n}{4} & 0\:\!0\:\!0\:\!0\:\!1\:\!0\:\!0\:\!0\:\!0\:\!1\:\!0\:\!0\:\!0\:\!0\:\!1\:\!0\:\!0\:\!0\:\!0\:\!1\:\!0\:\!0\:\!0\:\!0\:\!1 &  5 &  5 \\
\textbinomseq{n}{5} &0\:\!0\:\!0\:\!0\:\!0\:\!  1\:\!1\:\!1\:\!1\:\!1\:\!2\:\!2\:\!2\:\!2\:\!2\:\!3\:\!3\:\!3\:\!3\:\!3\:\!4\:\!4\:\! 4
   4   4 &  25 &  6 \\
\textbinomseq{n}{6} &0\:\!0\:\!0\:\!0\:\!0\:\!0\:\!  1\:\!2\:\!3\:\!4\:\!0\:\!2\:\!4\:\!1\:\!3\:\!0\:\!3\:\!1\:\!4\:\!2\:\!0\:\!4\:\!3  \:\! 2
   1 &  25 &  7 \\
\textbinomseq{n}{7} & 0\:\!0\:\!0\:\!0\:\!0\:\!0\:\!0\:\!1\:\!3\:\!1\:\!0\:\!0\:\!2\:\!1\:\!2\:\!0\:\!0\:\!3\:\!4\:\!3\:\!0\:\!0\:\!4\:\!2\:\!4 &  25 &  8 \\
\textbinomseq{n}{8} &0\:\!0\:\!0\:\!0\:\!0\:\!0\:\!0\:\!0\:\!1\:\!4\:\!0\:\!0\:\!0\:\!2\:\!3\:\!0\:\!0\:\!0\:\!3\:\!2\:\!0\:\!0\:\!0\:\!4\:\!1&  25 & 9  \\
\textbinomseq{n}{9} & 0\:\!0\:\!0\:\!0\:\!0\:\!0\:\!0\:\!0\:\!0\:\!1\:\!0\:\!0\:\!0\:\!0\:\!2\:\!0\:\!0\:\!0\:\!0\:\!3\:\!0\:\!0\:\!0\:\!0\:\!4& 27 & 10 \\
\textbinomseq{n}{10} & 0\:\!0\:\!0\:\!0\:\!0\:\!0\:\!0\:\!0\:\!0\:\!0\:\!1\:\!1\:\!1\:\!1\:\!1\:\!3\:\!3\:\!3\:\!3\:\!3\:\!1\:\!1\:\!1\:\!1\:\!1
& 27 & 11 \\
\textbinomseq{n}{11} & 0\:\!0\:\!0\:\!0\:\!0\:\!0\:\!0\:\!0\:\!0\:\!0\:\!0\:\!1\:\!2\:\!3\:\!4\:\!0\:\!3\:\!1\:\!4\:\!2\:\!0\:\!1\:\!2\:\!3\:\!4
& 27 & 12 \\
\textbinomseq{n}{12} & 0\:\!0\:\!0\:\!0\:\!0\:\!0\:\!0\:\!0\:\!0\:\!0\:\!0\:\!0\:\!1\:\!3\:\!1\:\!0\:\!0\:\!3\:\!4\:\!3\:\!0\:\!0\:\!1\:\!3\:\!1& 27 & 13 \\
\textbinomseq{n}{13} & 0\:\!0\:\!0\:\!0\:\!0\:\!0\:\!0\:\!0\:\!0\:\!0\:\!0\:\!0\:\!0\:\!1\:\!4\:\!0\:\!0\:\!0\:\!3\:\!2\:\!0\:\!0\:\!0\:\!1\:\!4 & 27 & 14 \\
\textbinomseq{n}{14} & 0\:\!0\:\!0\:\!0\:\!0\:\!0\:\!0\:\!0\:\!0\:\!0\:\!0\:\!0\:\!0\:\!0\:\!1\:\!0\:\!0\:\!0\:\!0\:\!3\:\!0\:\!0\:\!0\:\!0\:\!1 & 27 & 15 \\
\textbinomseq{n}{15} & 0\:\!0\:\!0\:\!0\:\!0\:\!0\:\!0\:\!0\:\!0\:\!0\:\!0\:\!0\:\!0\:\!0\:\!0\:\! 1\:\!1\:\!1\:\!1\:\!1\:\!4\:\!4\:\!4\:\!4\:\!4 & 27 & 16 \\
\textbinomseq{n}{16} & 0\:\!0\:\!0\:\!0\:\!0\:\!0\:\!0\:\!0\:\!0\:\!0\:\!0\:\!0\:\!0\:\!0\:\!0\:\!0\:\!1\:\!2\:\!3\:\!4\:\!0\:\!4\:\!3\:\!2\:\!1
 & 27 & 17 \\
\textbinomseq{n}{17} & 0\:\!0\:\!0\:\!0\:\!0\:\!0\:\!0\:\!0\:\!0\:\!0\:\!0\:\!0\:\!0\:\!0\:\!0\:\!0\:\!0\:\!1\:\!3\:\!1\:\!0\:\!0\:\!4\:\!2\:\!4 & 27 & 18 \\
\textbinomseq{n}{18} & 0\:\!0\:\!0\:\!0\:\!0\:\!0\:\!0\:\!0\:\!0\:\!0\:\!0\:\!0\:\!0\:\!0\:\!0\:\!0\:\!0\:\!0\:\!1\:\!4\:\!0\:\!0\:\!0\:\!4\:\!1& 25 & 19 \\
\textbinomseq{n}{19} & 0\:\!0\:\!0\:\!0\:\!0\:\!0\:\!0\:\!0\:\!0\:\!0\:\!0\:\!0\:\!0\:\!0\:\!0\:\!0\:\!0\:\!0\:\!0\:\!1\:\!0\:\!0\:\!0\:\!0\:\!4& 25 & 20 \\
\textbinomseq{n}{20} & 0\:\!0\:\!0\:\!0\:\!0\:\!0\:\!0\:\!0\:\!0\:\!0\:\!0\:\!0\:\!0\:\!0\:\!0\:\!0\:\!0\:\!0\:\!0\:\!0\:\! 1\:\!1\:\!1\:\!1\:\!1
 & 25 & 21 \\
\textbinomseq{n}{21} & 0\:\!0\:\!0\:\!0\:\!0\:\!0\:\!0\:\!0\:\!0\:\!0\:\!0\:\!0\:\!0\:\!0\:\!0\:\!0\:\!0\:\!0\:\!0\:\!0\:\!0\:\!1\:\!2\:\!3\:\!4 & 25 & 22 \\
\textbinomseq{n}{22} & 0\:\!0\:\!0\:\!0\:\!0\:\!0\:\!0\:\!0\:\!0\:\!0\:\!0\:\!0\:\!0\:\!0\:\!0\:\!0\:\!0\:\!0\:\!0\:\!0\:\!0\:\!0\:\!1\:\!3\:\!1 & 25 & 23 \\
\textbinomseq{n}{23} & 0\:\!0\:\!0\:\!0\:\!0\:\!0\:\!0\:\!0\:\!0\:\!0\:\!0\:\!0\:\!0\:\!0\:\!0\:\!0\:\!0\:\!0\:\!0\:\!0\:\!0\:\!0\:\!0\:\!1\:\!4 & 25 & 24 \\
\textbinomseq{n}{24} & 0\:\!0\:\!0\:\!0\:\!0\:\!0\:\!0\:\!0\:\!0\:\!0\:\!0\:\!0\:\!0\:\!0\:\!0\:\!0\:\!0\:\!0\:\!0\:\!0\:\!0\:\!0\:\!0\:\!0\:\!1 & 25 & 25 \\
\hline
\end{array}
$$
%\end{table}

%%=============================================%%
%% For submissions to Nature Portfolio Journals %%
%% please use the heading ``Extended Data''.   %%
%%=============================================%%

%%=============================================================%%
%% Sample for another appendix section			       %%
%%=============================================================%%

%% \section{Example of another appendix section}\label{secA2}%
%% Appendices may be used for helpful, supporting or essential material that would otherwise 
%% clutter, break up or be distracting to the text. Appendices can consist of sections, figures, 
%% tables and equations etc.

\end{appendices}

\end{document}